\documentclass[12pt, a4paper]{amsart}
\usepackage{amsmath}
\usepackage{geometry,amsthm, graphics,tabularx,amssymb,shapepar}
\usepackage[cmyk]{xcolor}
\usepackage{appendix}
\usepackage{amscd}
\usepackage[all]{xypic}
\usepackage{ulem}
\usepackage{bm}
\usepackage{extarrows}
\makeatletter
\newcommand*{\rom}[1]{\expandafter\@slowromancap\romannumeral #1@}
\makeatother

\newcommand{\BC}{{\mathbb {C}}}

\newcommand{\BN}{{\mathbb {N}}}

\newcommand{\BR}{{\mathbb {R}}}

\newcommand{\CF}{{\mathcal {F}}}
\newcommand{\CG}{{\mathcal {G}}}

\newcommand{\CO}{{\mathcal {O}}}

\newcommand{\RC}{{\mathrm {C}}}
\newcommand{\RD}{{\mathrm {D}}}

\newcommand{\Ri}{{\mathrm{i}}}

\newcommand{\RS}{{\mathrm {S}}}

\newcommand{\RU}{{\mathrm {U}}}

\newcommand{\Ad}{{\mathrm{Ad}}}

\newcommand{\Hom}{{\mathrm{Hom}}}

\newcommand{\Lie}{{\mathrm{Lie}}}

\newcommand{\wo}{{\widetilde\otimes}}

\newcommand{\Spec}{{\mathrm{Spec}}}

\newcommand{\wt}{\widetilde}
\newcommand{\wh}{\widehat}

\newcommand{\ad}{\operatorname{ad}}

\newcommand{\Ev}{\operatorname{Ev}}

\newcommand{\g}{\mathfrak g}

\newcommand{\q}{\mathfrak q}

\renewcommand{\l}{\mathfrak l}

\newcommand{\m}{\mathfrak m}

\newcommand{\C}{\mathbb{C}}
\newcommand{\R}{\mathbb R}

\newcommand{\abs}[1]{\lvert#1\rvert}

\newcommand{\la}{\langle}
\newcommand{\ra}{\rangle}

\newcommand{\be}{\begin {equation}}
\newcommand{\ee}{\end {equation}}
\newcommand{\bee}{\begin {equation*}}
\newcommand{\eee}{\end {equation*}}

\newcommand{\qaq}{\quad\textrm{and}\quad}

\newcommand{\cf}{\textit{cf.}~}

\renewcommand{\mid}{\,:\,}

\theoremstyle{Theorem}

\theoremstyle{Theorem}

\theoremstyle{Theorem}

\theoremstyle{Theorem}

\theoremstyle{Plain}

\theoremstyle{remark}

\theoremstyle{remark}

\theoremstyle{Definition}
\newtheorem{dfn}{Definition}[section]

\newtheorem{cord}[dfn]{Corollary}
\newtheorem{prpd}[dfn]{Proposition}
\newtheorem{thmd}[dfn]{Theorem}
\newtheorem{lemd}[dfn]{Lemma}

\newtheorem{remarkd}[dfn]{Remark}
\numberwithin{equation}{section}
\begin{document}

\title[Modules of Lie pairs and representations of formal Lie groups]{ 
Representations of formal Lie groups and Lie pairs}

\author[F. Chen]{Fulin Chen}
\address{School of Mathematical Sciences, Xiamen University,
 Xiamen, 361005, China} \email{chenf@xmu.edu.cn}

\author[B. Sun]{Binyong Sun}
\address{Institute for Advanced Study in Mathematics and New Cornerstone Science Laboratory, Zhejiang University,  Hangzhou, 310058, China}
\email{sunbinyong@zju.edu.cn}

\author[C. Wang]{Chuyun Wang}
\address{School of Mathematics and Statistics, Hainan University, Haikou, 570228, China}
\email{chuunw@amss.ac.cn} 


\begin{abstract}
Formal Lie groups, which generalize Lie groups in differential geometry, are analogous to formal group schemes in algebraic geometry.
In a previous paper, we established the basic theory of formal Lie groups, including a formal Lie theory theorem that identifies formal Lie groups with Lie pairs.  In this paper, we develop the foundations of the representation theory of formal Lie groups and Lie pairs. 
In particular, we construct modules of Lie pairs on function spaces over formal manifolds and prove that the category of representations of a formal Lie group is isomorphic to the category of modules of the corresponding Lie pair.
\end{abstract}

\maketitle

\tableofcontents

\section{Introduction and the main results}
\subsection{Motivations} 
 We begin by introducing the concept of Lie pairs.
\begin{dfn}\label{df:liepair}
	A \textbf{Lie pair} is a pair $(\mathfrak{q},L)$ consisting of a finite-dimensional complex Lie algebra $\mathfrak{q}$ and a real Lie group $L$, together with
	a   representation
	\[
	\Ad:\ L\curvearrowright \mathfrak{q}, \qquad (g,\eta )\mapsto \Ad_g \eta
	\]
	of $L$ on $\mathfrak{q}$ as Lie algebra automorphisms, and an injective Lie algebra homomorphism
	\[
	\iota:\ \mathfrak{l}\rightarrow \mathfrak{q}, \quad (\text{$\mathfrak{l}$ is the complexified Lie algebra of $L$}),
	\]
	subject to the following conditions:
	\begin{itemize}
		\item
		$\iota$ is $L$-equivariant, where $\mathfrak{l}$ carries the adjoint representation of $L$, and $\mathfrak{q}$ carries the representation $\Ad: L\curvearrowright\mathfrak{q}$; and
		\item
		the differential $\ad: \mathfrak{l}\curvearrowright \mathfrak{q}$ of $\Ad: L\curvearrowright \mathfrak{q}$ equals the action
		\[
		\mathfrak{l}\curvearrowright \mathfrak{q},\qquad (\tau,\eta)\mapsto [\iota(\tau),\eta].
		\]
	\end{itemize}
\end{dfn}
Throughout this paper, all Lie groups are assumed to be Hausdorff (and thus they are paracompact). They may or may not have countably many connected components. 
With the notation and assumptions of Definition \ref{df:liepair}, we refer to the quadruple $(\mathfrak{q}, L, \Ad, \iota)$, or simply the pair $(\mathfrak{q}, L)$ when $\Ad$ and $\iota$ are understood, as a Lie pair.

Every finite-dimensional complex Lie algebra $\mathfrak{q}$ determines a Lie pair $(\mathfrak{q},\{e\})$, while every real Lie group $L$ determines a Lie pair $(\mathfrak{l},L)$. Here and henceforth, the notation $\{e\}$ denotes the trivial group, and we use the corresponding lowercase Gothic letter to denote the complexified Lie algebra of a real Lie group.

Let $(\q,L)$ be a Lie pair.
In this paper, an LCS is a locally convex topological vector space over $\C$, not necessarily Hausdorff. However, every complete or quasi-complete LCS is assumed to be Hausdorff. 
With this convention, we define $(\q,L)$-modules as follows.
\begin{dfn}\label{def:ofqLmod}
	 A $(\q,L)$-\textbf{module} is a pair $(\mu,E)$, where $E$ is  a  quasi-complete LCS and $\mu:=(\mu_{\q},\mu_L)$ consists of a continuous Lie algebra action \[\mu_\q:\ \q\curvearrowright E, \quad (\eta,u)\mapsto \eta.u,\]
	and a continuous group action \[\mu_L:\ L\curvearrowright E, \quad (g,u)\mapsto g.u, \] such that 
	\begin{itemize}
		\item  for all $g\in L, \eta \in \q$ and $ u\in E$, \be\label{eq:qLmodule1} g.(\eta.(g^{-1}.u))=(\Ad_g \eta).u;\ee
		\item for all $\tau \in\Lie(L)$ and $ u\in E$, 
  \be\label{eq:qLmodule2} \lim_{t\rightarrow0}\frac{\exp (t\tau).u-u}{t}=\iota(\tau).u. \ee
  \end{itemize}
\end{dfn}
Here and henceforth,  $\Lie(H)$ is the real Lie algebra of a Lie group $H$, and $\exp: \Lie(H)\rightarrow H $ is the exponential map. 

\begin{dfn}
	Let $(\mu_1,E_1)$, $(\mu_2,E_2)$ be two $(\q,L)$-modules. A continuous linear map $\phi: E_1\rightarrow E_2$ is said to be  $(\q,L)$-equivariant if it is both $\q$-equivariant and $L$-equivariant. 
\end{dfn}
All $(\q, L)$-modules, together with all $(\q, L)$-equivariant continuous linear maps, form a category. 
In particular, the category of $(\q,\{e\})$-modules is the same as  the category of (continuous) $\q$-modules, whereas the category of $(\l,L)$-modules is the same as the category of smooth representations of $L$.


In the literature, representations of Lie pairs $(\mathfrak{q},L)$ have primarily been studied under the assumption that $L$ is compact (see \cite{KV}). However, for applications to the theory of automorphic forms, it is desirable to develop a theory of smooth representations for general Lie pairs $(\mathfrak{q},L)$, where $L$ is not necessarily compact. 

Our previous work \cite{CSW1, CSW2, CSW3, CSW4} laid the foundations of the theory of formal manifolds and formal Lie groups, extending the classical theory of smooth manifolds and Lie groups. This theory provides a natural framework for the study of smooth representations of general Lie pairs: general Lie pairs correspond to formal Lie groups (see \cite[Theorem 1.11]{CSW4}), and their smooth representations can be naturally realized on function spaces over formal manifolds.

In this paper, we introduce representations of formal Lie groups and prove that the category of representations of a formal Lie group is isomorphic to the category of modules of the corresponding Lie pair (see Theorem \ref{thm:eqMR}). We also construct modules of Lie pairs on the function spaces over formal manifolds studied in \cite{CSW2} (see Proposition \ref{prop:modofgG}), which will be used to construct projective and injective resolutions of modules of Lie pairs.

\subsection{Basics on formal Lie groups}\label{sec:forLie}
We first recall some notions concerning formal manifolds and formal Lie groups from \cite{CSW1, CSW4}.

For a smooth manifold $N$ and $k\in \BN:=\{0,1,2,\dots\}$, denote by $N^{(k)}$ the locally ringed space $(N, \CO_N^{(k)})$ over $\mathrm{Spec}(\BC)$.
 Here, $\CO_N^{(k)}$ is the sheaf 
\[ U\mapsto
\CO_N^{(k)}(U):=\RC^\infty(U)[[y_1,y_2,\dots,y_k]]
\] (the restriction maps are the obvious ones), 
where $\RC^\infty(U)$ is the $\BC$-algebra of complex-valued smooth functions on $U$, and $\RC^\infty(U)[[y_1,y_2,\dots,y_k]]$ is the $\BC$-algebra of formal power series with coefficients in $\RC^\infty(U)$.
\begin{dfn}\label{def:formalmanifold}
	A \textbf{formal manifold} is a locally ringed space $(M, \CO_M)$ over $\mathrm{Spec}(\BC)$ such that
	\begin{itemize}
		\item the topological space $M$ is paracompact and Hausdorff; and
		\item for every $a\in M$, there is an open neighborhood $U$ of $a$ in $M$ and $n,k\in \BN$ such that $(U, \CO_M|_U)$ is isomorphic to $(\R^n)^{(k)}$ as locally ringed spaces over $\mathrm{Spec}(\BC)$.
	\end{itemize}
\end{dfn}
By abuse of notation, we will often not distinguish a formal manifold $(M, \CO_M)$ from its underlying topological space $M$, and call $\CO_M$ the structure sheaf of it. 
An element of $\CO_M(M)$ is called a formal function on $M$.
Let $\m_{\CO_M}$ be the ideal sheaf defined by
\be\label{eq:defmo}
\m_{\CO_M}(U):=\{f\in \CO_M(U)\mid f_a\in \m_{M,a} \textrm{ for all $a\in U$}\},
\ee
where $f_a$ is the germ of $f$ at $a$, and $\m_{M,a}$ is the maximal
ideal of the stalk $\CO_{M,a}$.
Then
\be \label{eq:underM}
\underline{M}:=(M,\underline{\CO_M}:=\CO_M/\m_{\CO_M})
\ee
is a smooth manifold, called the reduction of $M$.


A morphism from a formal manifold $(M_1,\CO_{M_1})$ to another formal manifold $(M_2,\CO_{M_2})$ is a pair $\varphi=(\overline\varphi, \varphi^*)$, where $\overline\varphi: M_1\rightarrow M_2$ is a continuous map and
\be\label{phin}
\varphi^*:\ \overline\varphi^{-1}\CO_{M_2}\rightarrow \CO_{M_1}
\ee
is a $\BC$-algebra sheaf homomorphism that induces local homomorphisms on the stalks.
For open subsets $U_2$ of $M_2$ and $U_1$ of $M_1$ such that $\overline\varphi(U_1)\subset U_2$,
we write
\[\varphi^*_{U_2,U_1}:\CO_{M_2}(U_2)\rightarrow \CO_{M_1}(U_1)\] for the homomorphism of $\C$-algebras induced by \eqref{phin}.
If there is no confusion, we will denote $\varphi^*_{U_2,U_1}$ by $\varphi^*_{U_2}$ or $\varphi^*$ for simplicity.

Finite products exist in the category of formal manifolds (see \cite[Theorem 1.9]{CSW1}).
Given two morphisms $\varphi_1: M\rightarrow M_1$ and $\varphi_2: M\rightarrow M_2$ of formal manifolds, when no confusion is possible, we still use  
   \be\label{eq:varphi1times2} \varphi_1\times \varphi_2:\ M\rightarrow M_1\times M_2\ee to denote the composition
   \[
   M\xrightarrow{\Delta_{M}}M\times M\xrightarrow{\varphi_1\times\varphi_2}M_1\times M_2. \]
   Here and henceforth, $\Delta_M$ denotes the diagonal morphism of $M$.


\begin{dfn}\label{df:Lirgroup}
	A \textbf{formal Lie group} is a triple $(G,m,\varepsilon)$, where $G$ is a formal manifold,
    \[m: G\times G\rightarrow G \qaq \varepsilon: \mathrm{Spec}(\C)\rightarrow G \]
      are morphisms of formal manifolds, satisfying the following conditions:
	\begin{itemize}
 \item the diagram \be \label{eq:chainofG}
		\begin{CD}
			G\times G\times G @> m \times \mathrm{id}_G >>  G\times G\\
			@V\mathrm{id}_G\times m VV           @V V mV\\
			G\times G @>m>>  G \\
		\end{CD} \qquad (\textrm{$\mathrm{id}$ indicates the identity morphism})
		\ee commutes;
		\item the diagram  
		\[
		\xymatrix{
			\Spec(\BC)\times G \ar[dr]_= \ar[r]^{\ \ \ \  \varepsilon\times \mathrm{id}_G}        &G\times G \ar[d]_{m} & G\times \Spec(\BC)\ar[l]_{\mathrm{id}_G\times \varepsilon\ \ \ } \ar[dl]^=\\
			& G 
		}
		\]
  commutes; and
		
		\item there is a morphism $i: G\rightarrow G$  such that
		the diagram
	    \be \label{eq:comofi}	\begin{CD}
	    	G @> \mathrm{id}_G\times i >>  G\times G@<\ i\times \mathrm{id}_G <<G\\
	    	@V VV           @Vm V V @VVV\\
	    	\Spec(\BC) @> \varepsilon >>  G @<\varepsilon<<	\Spec(\BC) \\
	    \end{CD} \ee
		commutes.
	\end{itemize}
	
\end{dfn}


When no confusion is possible, we will not distinguish a formal Lie group $(G, m, \varepsilon)$ from $G$. The morphisms $m$ and $\varepsilon$ are called the multiplication morphism and the unit morphism, respectively. Let $e$ denote the image of the unique point in $\Spec (\BC)$ under the morphism $\varepsilon$, to be called the identity element of $G$.
As in the case of abstract groups, a morphism $i: G\rightarrow G$ satisfying the commutative diagram \eqref{eq:comofi} is unique. The morphism $i$ satisfies $i\circ i=\mathrm{id}_G$, and it is called the inversion morphism. 


The reduction $\underline{G}$ of a formal Lie group $G$ is naturally a Lie group (see \cite[Example 3.5]{CSW4}).
\subsection{Hopf formal algebras and Hopf  formal coalgebras}\label{sec:comodofHopf}
For two LCS $E$ and $F$, the algebraic tensor product $E \otimes F$ admits two natural locally convex topologies: the inductive tensor product $E \otimes_{\mathrm{i}} F$ and the projective tensor product $E \otimes_{\pi} F$. These are characterized by universal properties for separately and jointly continuous bilinear maps, respectively, just as the algebraic tensor product $E \otimes F$ is characterized by bilinear maps. We denote the quasi-completions of the maximal Hausdorff quotients of these two topological tensor product spaces by
\[
E\wt\otimes_{\mathrm{i}}F
\quad\text{and}\quad
E\wt\otimes_{\pi}F,\]
respectively, and the completions of these quotients by
\[
E\widehat\otimes_{\mathrm{i}}F
\quad\text{and}\quad
E\widehat\otimes_{\pi}F,
\]
respectively.

Given two continuous linear maps $\phi_1:E_1\to F_1$ and $\phi_2:E_2\to F_2$ between LCS, we denote by $\phi_1\otimes\phi_2$ the continuous linear map on the various topological tensor products (or their completions) obtained by the tensor product of $\phi_1$ and $\phi_2$. For example, we have the maps
\[
\phi_1\otimes\phi_2:\ E_1\otimes_\pi E_2 \longrightarrow F_1\otimes_\pi F_2 \qaq
\phi_1\otimes\phi_2:\ E_1\wt{\otimes}_\Ri E_2 \longrightarrow F_1\wt{\otimes}_\Ri F_2.
\]

A {$\wt{\otimes}_{\pi}$-algebra}  is a quasi-complete  LCS $A$ together with a continuous associative multiplication $A \wt{\otimes}_{\pi} A \rightarrow A$  and a unit $\mathbb{C} \rightarrow A$. Similarly, we have the notions of  $\wt{\otimes}_{\mathrm{i}}$-coalgebras,  Hopf $\wt{\otimes}_{\mathrm{i}}$-algebras, and Hopf  $\wt{\otimes}_{\pi}$-algebras (\cf \cite[Appendix A.2]{CSW4}).  

Given a formal manifold $(M,\mathcal{O}_M)$, equip $\mathcal{O}_M(M)$ with the smooth topology (see \cite[Definition 4.1]{CSW1}). Then, together with the unit
\begin{equation}\label{eq:unitofOM}
\mathbb{C} \rightarrow \mathcal{O}_M(M),\quad c \mapsto c
\end{equation}
and the continuous multiplication
\begin{equation}\label{eq:mulofOM}
\mathcal{O}_M(M) \wt{\otimes}_{\pi} \mathcal{O}_M(M) \rightarrow \mathcal{O}_M(M),\quad f_1 \otimes f_2 \mapsto f_1 f_2 \quad (f_1, f_2 \in \mathcal{O}_M(M)),
\end{equation}
the LCS $\mathcal{O}_M(M)$ becomes a $\wt{\otimes}_{\pi}$-algebra (\cf \cite[Proposition 4.8]{CSW1}).

Recall from \cite[Section 5]{CSW2} the LCS $\RD_c^{-\infty}(M;\mathcal{O}_M)$ of compactly supported formal distributions on $M$. By \cite[Example 4.2]{CSW1} and \cite[Theorem 1.1]{CSW2}, $\CO_M(M)$ and $\RD_c^{-\infty}(M;\mathcal{O}_M)$ generalize, respectively, the spaces of (complex-valued) smooth functions and compactly supported distributions on a smooth manifold.

By \cite[Proposition 5.14 and Lemma A.12]{CSW2}, there are  identifications
\be \label{eq:OM'}
\RD_c^{-\infty}(M;\mathcal{O}_M) = (\mathcal{O}_M(M))' 
\ee
and
\be \label{eq:DMM=DMDM}
\RD_c^{-\infty}(M;\mathcal{O}_M) \wt{\otimes}_{\mathrm{i}} \RD_c^{-\infty}(M;\mathcal{O}_M) = (\mathcal{O}_M(M) \wt{\otimes}_{\pi} \mathcal{O}_M(M))'
\ee
of LCS. Then, the LCS $\RD_c^{-\infty}(M;\mathcal{O}_M)$ becomes an  $\wt{\otimes}_{\mathrm{i}}$-coalgebra: the comultiplication is 
the transpose
of \eqref{eq:mulofOM}, and the counit is the transpose
of \eqref{eq:unitofOM}.  Here and henceforth, for any LCS $E$,  $E'$ is the space of all continuous linear functionals on $E$,  equipped with the strong topology (see \cite[Chapter 19]{Tr}).


We make the following definitions (\cf \cite[Definitions 1.9 and 1.10]{CSW4}).

\begin{dfn}\label{def:foralg}

\noindent(a) A  $\wt{\otimes}_{\pi}$-algebra is called a \textbf{formal algebra} if it is isomorphic to $\mathcal{O}_M(M)$ as a  $\wt{\otimes}_{\pi}$-algebra for some formal manifold $M$.
    
\noindent(b) An  $\wt{\otimes}_{\mathrm{i}}$-coalgebra is called a \textbf{formal coalgebra} if it is isomorphic to $\RD_c^{-\infty}(M;\mathcal{O}_M)$ as  a $\wt{\otimes}_{\mathrm{i}}$-coalgebra for some formal manifold $M$.
\end{dfn}

\begin{dfn}\label{def:formalHopfalg}

\noindent(a) A Hopf $\wt{\otimes}_{\pi}$-algebra is called a \textbf{Hopf formal algebra} if it is a formal algebra as a $\wt{\otimes}_{\pi}$-algebra.
    
\noindent(b) A Hopf $\wt{\otimes}_{\mathrm{i}}$-algebra is called a \textbf{Hopf formal coalgebra} if it is a formal coalgebra as an $\wt{\otimes}_{\mathrm{i}}$-coalgebra.
\end{dfn}

\begin{dfn} Let $A$ be a Hopf formal algebra.  

\noindent (a) An $A$-\textbf{comodule} is a pair $(\varrho, E)$
consisting of a quasi-complete LCS  $E$ and a continuous linear map \[\varrho:\  E\rightarrow E\wt\otimes_{\pi} A,\]
such that the diagrams
\be\label{eq:RepchianruleofO}\xymatrix{
E
\ar[r]^-{\varrho}
\ar[d]_{\varrho}
&
E\wt\otimes_{\pi} A
\ar[d]^{\varrho \otimes \mathrm{id}_{A}}
\\
E\wt\otimes_{\pi} A
\ar[r]^-{\mathrm{id}_{E}\otimes \Delta}
&
E\wt\otimes_{\pi} A\wt\otimes_{\pi}A
} \qaq \xymatrix{E \ar[r]^{\varrho} \ar[rd]_{=} &E\wt\otimes_{\pi}A\ar[d]^{\mathrm{id}_E\otimes\varepsilon} \\ &E\wt\otimes_{\pi}\BC}\ee 
are commutative, where $\Delta$ is the comultiplication of  $A$, and $\varepsilon$ is the counit of $A$. 

\noindent(b) Let $(\varrho_1, E_1)$ and $(\varrho_2, E_2)$ be two comodules of $A$.
    A continuous linear map $\phi: E_1\rightarrow E_2$ is said to be  $A$-equivariant if the diagram 
    \[\begin{CD}
        E_1 @>\varrho_1>> E_1\wt\otimes_{\pi}A\\ @V\phi VV @VV\phi\otimes\mathrm{id}_{A} V \\ E_2 @>\varrho_2>> E_2\wt\otimes_{\pi}A
    \end{CD}\] is commutative. 
\end{dfn}
All comodules of a Hopf formal algebra $A$, together with all $A$-equivariant continuous linear maps, form a category. 

For every Hopf formal coalgebra $C$, let
\be\label{eq:LC}
\CG(C):=\{\eta\in C\mid \Delta(\eta)=\eta\otimes\eta, \eta\neq0\}
\ee be the set of group-like elements in $C$,
where $\Delta$ is the comultiplication of $C$. Then the multiplication and the unit element of $C$, endow $\CG(C)$ with a group structure. Equipped with the subspace topology inherited from $C$, the group $\CG(C)$ is a topological group (see Corollary \ref{cor:Gembed}).



\begin{dfn}\label{def:Cmod}
Let $C$ be a  Hopf formal coalgebra.

\noindent(a) A $C$-\textbf{module} is a pair $(\pi,E)$
consisting of a quasi-complete LCS $E$ and a continuous linear map
\[\pi:\  C\wt\otimes_{\mathrm{i}} E\rightarrow E,\quad (\eta\otimes u)\mapsto \eta.u:=\pi(\eta\otimes u)\] 
such that the diagrams
\be\label{eq:RepchianruleofD}
 \xymatrix{ C\otimes_{\Ri} C\otimes_{\Ri} E \ar[r]^-{\mathrm{id}_{C}\otimes \pi} \ar[d]_{m\otimes \mathrm{id}_E} & C\wt\otimes_{\Ri} E \ar[d]^{\pi} \\ C\wt\otimes_{\Ri} E \ar[r]^{\pi} & E } \qaq \xymatrix{\BC\wt\otimes_\Ri E\ar[r]^{\epsilon\otimes \mathrm{id}_E}\ar[rd]_{=} & C\wt\otimes_{\Ri} E\ar[d]^{\pi} \\ & E}\ee commute,  where $m$ is the multiplication of $C$, and $\epsilon$ is the unit of $C$. 

\noindent (b)
A $C$-module $(\pi,E)$ is said to be \textbf{continuous} if the induced map 
\[ \CG(C)\times E\rightarrow E,\quad (g,u)\mapsto g.u,\]  is continuous. 

\noindent (c) Let $(\pi_1, E_1)$ and $(\pi_2, E_2)$ be two continuous modules of $C$.
    A continuous linear map $\phi: E_1\rightarrow E_2$ is said to be  $C$-equivariant if the diagram 
    \[\begin{CD}
        C\wt\otimes_{\mathrm{i}} E_1 @>\pi_1>> E_1\\ @V\mathrm{id}_{C}\otimes\phi VV @VV\phi V \\ C\wt\otimes_{\mathrm{i}} E_2@>\pi_2>> E_2
    \end{CD}\]
    commutes.
\end{dfn}
All continuous modules of $C$, together with all $C$-equivariant continuous linear maps, form a category.
\begin{remarkd}

The completed (resp.\, quasi-completed) projective tensor product is associative, whereas the completed (resp.\, quasi-completed)  inductive tensor product need not be (see \cite[discussion following Theorem~8.12]{G} for example). We therefore use the uncompleted inductive tensor product in the first diagram of \eqref{eq:RepchianruleofD}.
\end{remarkd}

\subsection{Main results and structure of the paper}

We first recall the formal Lie theory theorem from \cite{CSW4}.
\begin{thmd}\label{thm:equiGqL}
    The following categories are equivalent to each other:
 \begin{enumerate}
  \item the category of Lie pairs;
  \item the category of Hopf formal  coalgebras; 
   \item the opposite category of the category of Hopf formal algebras; and
   \item the category of formal Lie groups.
\end{enumerate}
\end{thmd}
  
Precisely, for every formal Lie group $(G,\CO_G)$, the space $\CO_G(G)$ is a Hopf formal algebra. By taking the dual, the space $\RD_c^{-\infty}(G;\CO_G)=(\CO_G(G))'$ is a Hopf formal coalgebra.
Moreover, together with the adjoint representation $\mathrm{Ad}:\underline G  \curvearrowright \g$ and the canonical Lie algebra embedding $\underline{\g}\rightarrow \g$,  the pair $(\g,\underline{G})$ is a Lie pair, where $\mathfrak{g}$ is the complex Lie algebra of $G$, and $\underline{\mathfrak{g}}$ is the complexified  Lie algebra of $\underline{G}$. The above constructions yield equivalences between the four categories in Theorem \ref{thm:equiGqL}. See \cite{CSW4} or  Section 2 for more details. 

Let $G$ be a formal Lie group. By Corollary \ref{cor:Gembed},
\be\label{eq:GtoEv} G\rightarrow \CG(\RD_c^{-\infty}(G;\CO_G)),\quad g\mapsto \Ev_g\ee is an isomorphism between topological groups, where $\Ev_g\in \RD_c^{-\infty}(G;\CO_G)=(\CO_G(G))'$ is the composition of 
\be\label{evg}
\CO_G(G)\xrightarrow{\textrm{taking the germ}} \CO_{G,g}\xrightarrow{\textrm{quotient map}} \CO_{G,g}/\m_{G,g}=\BC, \ee 
and $\m_{G,g}$ is the maximal ideal of the  stalk $\CO_{G,g}$. 
In view of this, we make the following definition. 
\begin{dfn}
A \textbf{representation of } $G$ is a continuous module of the Hopf formal coalgebra $\RD_c^{-\infty}(G;\CO_G)$. Equivalently, it is a $\RD_c^{-\infty}(G;\CO_G)$-module $(\pi,E)$ such that the induced map 
\[G\times E\rightarrow E,\quad (g,u)\mapsto\Ev_g.u\]
is continuous. \end{dfn}

By \eqref{eq:GtoEv} and \eqref{eq:barreled}, the following remark is obvious. \begin{remarkd}Every $\RD_c^{-\infty}(G;\CO_G)$-module $(\pi,E)$ with $E$ barreled is a representation of $G$.
\end{remarkd}

The following theorem identifies representations of $G$ with modules of the corresponding Lie pair $(\g,\underline{G})$ and comodules of the corresponding Hopf formal algebra $\CO_G(G)$.

\begin{thmd}\label{thm:eqMR}
Let $G$ be a formal Lie group. Then the following categories are isomorphic to each other:
\begin{enumerate}
    \item the category of $(\mathfrak{g},\underline{G})$-modules;
    \item the category of comodules of the Hopf formal algebra $\CO_G(G)$; and
    \item the category of representations of $G$.
\end{enumerate}
\end{thmd}




The remainder of this paper is organized as follows. 
Section~2 recalls the notions of Hopf formal algebras, Hopf formal coalgebras, and Lie pairs associated with formal Lie groups. 
Section~3 provides explicit descriptions of the Hopf formal algebra and the Hopf formal coalgebra corresponding to a Lie pair $(\mathfrak{q}, L)$. 
Section~4 constructs modules of Lie pairs on function spaces over formal manifolds. 
Finally, Section~5 is devoted to the proof of Theorem~\ref{thm:eqMR}.


\section{Associated structures of formal Lie groups}
\subsection{Hopf formal algebras and Hopf formal coalgebras}
 In this subsection, we recall the Hopf formal algebra and the Hopf formal coalgebra associated with a formal Lie group.

Recall from Section \ref{sec:comodofHopf} that for each formal manifold $M$, the LCS $\CO_M(M)$ is a formal algebra, and the LCS $\RD^{-\infty}_c(M;\CO_M)$ is a formal coalgebra. 

For each morphism \[\varphi=(\overline\varphi, \varphi^*):\ (M_1,\CO_{M_1})\rightarrow (M_2,\CO_{M_2}),\]the homomorphism $\varphi^*: \CO_{M_2}(M_2)\rightarrow \CO_{M_1}(M_1)$ is a continuous homomorphism between formal algebras (see \cite[Theorem 1.8]{CSW1}), and the transpose \be \label{eq:tvarphi*}{}^t\varphi^*: \ \RD_c^{-\infty}(M_1;\CO_{M_1})\rightarrow \RD_c^{-\infty}(M_2;\CO_{M_2})\ee of $\varphi^*$ is a continuous homomorphism between formal coalgebras.

By \cite[Theorem 5.11]{CSW1} and \cite[the proof of Proposition 9.2]{CSW4}, we have the following result.
\begin{prpd}\label{prop:equiMO} The following categories are equivalent to each other:
\begin{enumerate}
	\item the category of formal manifolds;
	\item the opposite category of the category of formal algebras; and 
	\item the category of formal coalgebras.
\end{enumerate}
\end{prpd}
 
In Proposition \ref{prop:equiMO}, the equivalence of (1) and (2) is given by the functor
\[M\mapsto \CO_M(M),\quad (\varphi:M_1\rightarrow M_2)\mapsto (\varphi^*: \CO_{M_2}(M_2)\rightarrow \CO_{M_1}(M_1)).\] 
The equivalence of (1) and (3) is given by the functor \[M\mapsto \RD^{-\infty}_c(M;\CO_M),\quad (\varphi: M_1\rightarrow M_2)\mapsto ({}^t\varphi^*:\RD^{-\infty}_c(M_1;\CO_{M_1})\rightarrow \RD^{-\infty}_c(M_2;\CO_{M_2})).\] 
Furthermore, 
the equivalence of (2) and (3) is given by the functors \[A\mapsto A', \quad (\phi: A_1\rightarrow A_2)\mapsto({}^t\phi: A_2'\rightarrow A_1') \]  and \[ C\mapsto C', \quad (\phi: C_1\rightarrow C_2)\mapsto({}^t\phi: C_2'\rightarrow C_1'). \] 
Throughout the paper, for a continuous linear map $\phi: E_1\rightarrow E_2$, the map 
\[{}^t\phi:\ E_2'\rightarrow E_1'\]
denotes the transpose of $\phi$.




Recall from \cite[Theorem 1.9]{CSW1} and \cite[Proposition 5.20]{CSW2} that for any two formal manifolds $M_1$ and $M_2$, \be \label{eq:O_3}\CO_{M_1\times M_2}(M_1\times M_2)=\CO_{M_1}(M_1)\wt\otimes_{\pi}\CO_{M_2}(M_2)=\CO_{M_1}(M_1)\wh\otimes_{\pi}\CO_{M_2}(M_2)\ee as formal algebras and \be\label{eq:D3}
\begin{split}
\RD^{-\infty}_c(M_1\times M_2;\CO_{M_1\times M_2})
&=\RD^{-\infty}_c(M_1;\CO_{M_1})\wt\otimes_{\Ri}
\RD^{-\infty}_c(M_2;\CO_{M_2})\\
&=\RD^{-\infty}_c(M_1;\CO_{M_1})\wh\otimes_{\Ri}
\RD^{-\infty}_c(M_2;\CO_{M_2})
\end{split}
\ee
 as formal coalgebras. 

 Let $G$ be a formal Lie group. Then we have the continuous homomorphisms 
 \be\label{eq:formalhopfalg}
\begin{aligned}
		&m^*:\CO_G(G)\rightarrow \CO_{G\times G}(G\times G)=\CO_G(G) \wt{\otimes}_{\pi}\CO_G(G)\quad\text{(see \eqref{eq:O_3})},\\ 
 	&\varepsilon ^*: \CO_G(G)\rightarrow \BC, \qaq   i^*: \CO_G(G)\rightarrow \CO_G(G)
\end{aligned}
\ee
between formal algebras, 
and the continuous homomorphisms 
\[{}^tm^*: \RD^{-\infty}_c(G;\CO_G)\wt\otimes_{\mathrm i}\RD^{-\infty}_c(G;\CO_G)=\RD^{-\infty}_c(G\times G;\CO_{G\times G})\rightarrow \RD^{-\infty}_c(G;\CO_G) \quad\text{(see \eqref{eq:D3})},\]
   \[ {}^t\varepsilon^*: \BC\rightarrow \RD^{-\infty}_c(G;\CO_G), \qaq {}^ti^*: \RD^{-\infty}_c(G;\CO_G)\rightarrow \RD^{-\infty}_c(G;\CO_G)
\]
between formal coalgebras. 
The following result follows from \cite[Proposition 9.1]{CSW4}.
\begin{prpd}\label{prop:formalliegrouptohopfalg0}
    Let $G$ be a formal Lie group. 
    
\noindent(a) The formal algebra $\CO_G(G)$ is a Hopf formal algebra together with the comultiplication $m^*$, the counit $\varepsilon^*$ and 
the antipode $i^*$.

\noindent(b) The formal coalgebra $\RD^{-\infty}_c(G;\CO_G)$ is a Hopf formal coalgebra,
together with the multiplication ${}^tm^*$, the unit ${}^t\varepsilon^*$ and 
the antipode ${}^ti^*$.  
\end{prpd}
Throughout this paper, for $\eta_1,\eta_2\in \RD^{-\infty}_c(G;\CO_G)$, set \be\label{eq:eta12} \eta_1\eta_2\,:=\, {}^tm^*(\eta_1\otimes\eta_2)\in \RD^{-\infty}_c(G;\CO_G). \ee

Recall the following result from \cite[Proposition 9.2]{CSW4}. 
\begin{prpd}\label{prop:formalliegrouptohopfalg} 
The following categories are equivalent to each other:
\begin{enumerate}
	\item the category of formal Lie groups;
	\item the opposite category of the category of Hopf formal algebras; and 
	\item the category of Hopf formal coalgebras.
\end{enumerate}
\end{prpd}
In Proposition \ref{prop:formalliegrouptohopfalg}, 
the equivalence of (1) and (2) is given by the functor
\[G\mapsto \CO_G(G),\quad (\varphi:G_1\rightarrow G_2)\mapsto (\varphi^*: \CO_{G_2}(G_2)\rightarrow \CO_{G_1}(G_1)).\] 
The equivalence of (1) and (3) is given by the functor \[G\mapsto \RD^{-\infty}_c(G;\CO_G),\quad (\varphi: G_1\rightarrow G_2)\mapsto ({}^t\varphi^*:\RD^{-\infty}_c(G_1;\CO_{G_1})\rightarrow \RD^{-\infty}_c(G_2;\CO_{G_2})).\] 
Furthermore, 
the equivalence of (2) and (3) is given by the functors \[A\mapsto A', \quad (\phi: A_1\rightarrow A_2)\mapsto({}^t\phi: A_2'\rightarrow A_1') \]  and \[ C\mapsto C', \quad (\phi: C_1\rightarrow C_2)\mapsto({}^t\phi: C_2'\rightarrow C_1'). \] 
\subsection{The Dirac embedding}
Let $(M,\CO_M)$ be a formal manifold. 
As in \eqref{evg}, for each $a\in M$, write $\Ev_a\in \RD^{-\infty}_c(M;\CO_M)=(\CO_M(M))'$ for  the composition of 
\be\label{eq:Ev_a}\Ev_a:\ \CO_M(M)\xrightarrow{\textrm{taking the germ}} \CO_{M,a}\xrightarrow{\textrm{quotient map}} \CO_{M,a}/\m_{M,a}=\C, \ee where $\m_{M,a}$ is the maximal ideal of $\CO_{M,a}$. 
Throughout this paper, for every $f\in \CO_M(M)$, let $f(a)$ denote the image of $f$ under the map $\Ev_a$.

It follows from \cite[Proposition 5.5 and Lemma 5.8]{CSW1} that the map 
\be\label{eq:Dirac} \Ev:\ M\rightarrow \RD^{-\infty}_c(M;\CO_M),\quad a\mapsto \Ev_a \ee
is injective. 

When $M$ is a smooth manifold, the map \eqref{eq:Dirac} is the usual
Dirac embedding (\cf \cite{Mi}).
\begin{prpd}\label{prop:dirac}
 The map \eqref{eq:Dirac} is a topological embedding.
\end{prpd}
\begin{proof}
  We first prove that the inverse map \[\Ev^{-1}: \ \Ev(M)\rightarrow M\] is continuous. Let $\{b_\alpha\}_{\alpha\in A}$ be a net in $M$ such that \[\lim_{\alpha\in A} \Ev_{b_\alpha}=\Ev_b \quad \text{(in the LCS $\RD^{-\infty}_c(M;\CO_M)$)}\] for some $b\in M$. 
Note that, for every bounded subset
$B\subset \CO_M(M)$, one has that
\[
\lim_{\alpha\in A}
\sup_{f\in B}
\left|\langle \Ev_{b_\alpha}-\Ev_b,f\rangle\right|
=
\lim_{\alpha\in A}
\sup_{f\in B}|f(b_\alpha)-f(b)|
=0.
\]
Then for each 
$f_0\in \CO_M(M)$,
\[\lim_{\alpha\in A} f_0(b_{\alpha
})=f_0(b)\quad \text{(replace $B$ by $\{f_0\}$)}.\] This implies that 
 $\lim_{\alpha\in A} b_{\alpha}=b$ in the underlying topological space $M$ (see \cite[Lemma 5.8]{CSW1}), as desired.
 
 We now prove the continuity of  \eqref{eq:Dirac}. 
 By \cite[Lemma 5.3]{CSW2}, we assume without loss of generality  that $M=(\BR^n)^{(k)}$ for some $n,k\in\BN$.
 Let $\{b_\alpha\}_{\alpha\in A}$ be a net in $M$ with $\lim_{\alpha\in A}b_\alpha= b$ in $M$. It suffices to prove that 
 \be \label{eq:toEvb}
\lim_{\alpha\in A}\Ev_{b_\alpha}= \Ev_b
\quad \text{in } \RD^{-\infty}_c(M;\CO_M),
\ee
namely, for each bounded subset $B\subset \CO_M(M)$,  \be\label{eq:toEvb2} \lim_{\alpha\in A}\sup_{f\in B}|\la \Ev_{b_{\alpha}}-\Ev_b,f\ra|=\lim_{\alpha\in A}\sup_{f\in B}|f(b_{\alpha})-f(b)|=0.\ee
 When $k=0$, \eqref{eq:toEvb2} follows from the mean value theorem in several variables (see \cite[Theorem 12.9]{Ap}). 
 This, together with \cite[Examples 3.11 and 4.4]{CSW1}, implies that 
 \eqref{eq:toEvb2} holds for arbitrary $k\in\BN$, as desired. This finishes the proof of the proposition.
\end{proof}
%

Let $G$ be a formal Lie group. Recall from \eqref{eq:LC} that $\CG(\RD^{-\infty}_c(G;\CO_G))$ is the group of group-like elements in $\RD^{-\infty}_c(G;\CO_G)$.
The following result follows from \cite[Proposition 5.5]{CSW1} and Proposition \ref{prop:formalliegrouptohopfalg}.
\begin{cord}\label{cor:Gembed}
Let $G$ be a formal Lie group. Then the image of the topological  embedding 
\be\label{eq:DiracG} G\rightarrow  \RD^{-\infty}_c(G;\CO_G),\quad g\mapsto \Ev_g \ee is $\CG(\RD^{-\infty}_c(G;\CO_G))$.  In particular, 
$\CG(\RD^{-\infty}_c(G;\CO_G))$ is a topological group isomorphic to $G$.
\end{cord}

\subsection{Complex Lie algebras of formal Lie groups} 

Let $G$ be a formal Lie group. We write $\mathrm{Prim}(\RD^{-\infty}_c(G;\CO_G))$ for the set of primitive elements of the Hopf formal coalgebra $\RD^{-\infty}_c(G;\CO_G)$, namely,
\be \label{eq:Prim}
\mathrm{Prim}(\RD^{-\infty}_c(G;\CO_G))=\{\eta\in \RD^{-\infty}_c(G;\CO_G):\ \Delta(\eta)=\eta\otimes \Ev_e+\Ev_e\otimes \eta\},
\ee  where $\Delta$ is the comultiplication of the Hopf formal coalgebra $\RD^{-\infty}_c(G;\CO_G)$. 
It is well-known that $\mathrm{Prim}(\RD^{-\infty}_c(G;\CO_G))$ forms a complex Lie algebra under the commutator bracket
\[(\eta_1,\eta_2)\mapsto [\eta_1,\eta_2]:=\eta_1\eta_2-\eta_2\eta_1 \quad\text{(see \cite[(2.19)]{CSW4})}. \]

\begin{dfn}
     The complex Lie algebra  $\mathrm{Prim}(\RD^{-\infty}_c(G;\CO_G))$ is called the \textbf{complex Lie algebra} of the formal Lie group $G$.
\end{dfn}
Here and in the following, 
the complex Lie algebra of the formal Lie group $G$ is denoted by $\Lie_{\BC}(G)$, or simply by $\g$ if there is no confusion.  
When $G$ is a Lie group,  $\Lie_{\BC}(G)$ is just the complexified Lie algebra of $G$. 

As usual, the universal enveloping algebra $\RU(\g)$ is a Hopf algebra. Equip $\RU(\g)$ with the finest locally convex topology.  Then $\RU(\g)$ becomes a Hopf formal coalgebra.
Moreover, the algebra homomorphism \be \label{eq:Utodist}\RU(\g)\rightarrow \RD^{-\infty}_c(G;\CO_G) \ee induced by the inclusion $\g\rightarrow \RD^{-\infty}_c(G;\CO_G)$
      is a topological embedding between Hopf formal coalgebras (see \cite[Proposition 2.8]{CSW4} and \cite[Section 5.3]{CSW2}).
In view of this, we will often identify $\RU(\g)$ as a subspace of $\RD^{-\infty}_c(G;\CO_G)$ 
without further explanation.

Let
\[
\varphi=(\bar\varphi,\varphi^*):\, G\rightarrow H
\] be a homomorphism between formal Lie groups. 
The homomorphism 
\[
{}^t\varphi^*:\, \RD_c^{-\infty}(G;\CO_{G})\rightarrow \RD_c^{-\infty}(H;\CO_{H})
\]
of Hopf formal coalgebras restricts to a Lie algebra homomorphism
\[
d\varphi_e:={}^t\varphi^*|_{\mathrm{Lie}_\C(G)}:\ \mathrm{Lie}_\C(G)
\rightarrow \mathrm{Lie}_\C(H), 
\] 
to be called the differential of $\varphi$ at $e$ (see \cite[Proposition 2.8]{CSW4}).


\subsection{Lie pairs associated to formal Lie groups}
For each formal manifold $(M,\CO_M)$ and its reduction $\underline{M}=(M,\underline{\CO_M})$ (see \eqref{eq:underM}), there is a canonical morphism 
\[\iota=(\overline{\iota},\iota^*): \ \underline M\rightarrow M\]
between formal manifolds, where $\overline{\iota}=\mathrm{id}_M$ is the identity map on the underlying topological space $M$, and for every open subset $U$ of $M$, $\iota^*_{U}$ is the homomorphism  \be \label{eq:iotaU}\iota^*_{U}:\ \CO_M(U)\rightarrow (\CO_M/\m_{\CO_M})(U)=\underline{\CO_M}(U).\ee 

Let 
$
\varphi:\ (M,\CO_M)\rightarrow (M',\CO_{M'})
$
be a morphism of formal manifolds. Then there is a unique smooth map
\be\label{eq:reductionvarphi}
\underline{\varphi}:\ \underline{M}\rightarrow \underline{M'}
\ee between smooth manifolds 
such that the diagram
\be \label{eq:underlinecommdiag}
\begin{CD}
	M @>  \varphi >> M'\\
	@A\iota AA          @AA\iota A\\
	\underline{M}@> \underline{\varphi} >>  \underline{M'}
\end{CD}
\ee
commutes. We call $\underline{\varphi}$ the reduction of $\varphi$.

Let $G$ be a formal Lie group. 
Recall from \cite[Example 3.5]{CSW4} that together with the multiplication morphism $\underline{m}$ and the unit morphism $\underline{\varepsilon}$, the reduction $\underline{G}$ becomes a Lie group.  
In addition, the morphism $\iota: \underline{G}\rightarrow G$ is a homomorphism between formal Lie groups (see \cite[Example 3.5]{CSW4}), and its differential at $e$ 
\[d\iota_e: \underline \g\rightarrow \g\] is an injective Lie algebra homomorphism (see \cite[Lemma 2.7]{CSW3}). Here, $\underline{\g}$ denotes the complexified  Lie algebra of $\underline{G}$. In view of this, we will often identify $\underline{\g}$ as a subspace of $\g$ 
without further explanation.

It follows from \cite[Proposition 5.8]{CSW4} that together with the adjoint representation  \[\Ad: \underline G\curvearrowright \g \] and the  injective homomorphism  $d\iota_e: \underline\g \rightarrow \g$, the pair $( \g, \underline G)$  is a Lie pair. For the precise description of the adjoint representation, see \cite[(5.12)]{CSW4} or Corollary \ref{cor:realAd}. 

We will use the following result from \cite[Proposition 9.4]{CSW4}
throughout this paper without further reference.
\begin{prpd}\label{prop:Gtopair}
    The functor \be\label{eq:Gtopair}G\mapsto (\g,\underline{G},\Ad,d\iota_e)\ee is an equivalence between the category of formal Lie groups and the category of Lie pairs.
\end{prpd}

\section{Associated structures of  Lie pairs}
In this section, we give explicit descriptions of the Hopf formal algebra and the Hopf formal coalgebra corresponding to a Lie pair (see Proposition \ref{prop:OofLLqLl}), which will be used to prove Theorem \ref{thm:eqMR} in Section 5.

\subsection{Topological Hopf algebras associated to $L\ltimes \CG_{\q}$}
 Let $(\q,L,\Ad,\iota)$ be a Lie pair. 
 Recall that  $\RU(\q)$ is a Hopf formal coalgebra. Then its strong dual $(\mathrm{U}(\mathfrak{q}))'$ is a  Hopf formal algebra. The formal Lie group corresponding to $\RU(\q)$ is denoted by \[\CG_{\mathfrak{q}}:=(\{e\},(\mathrm{U}(\mathfrak{q}))').\] Here, we do not distinguish between a sheaf over a singleton and its space of global sections. It is clear that 
  \[\CO_{\CG_{\q}}(\{e\})=(\RU(\q))'\] as Hopf formal algebras and 
  \[\RD^{-\infty}_c(\{e\};\CO_{\CG_{\q}})=\RU(\q)\] as Hopf formal coalgebras 
  (\cf  \cite[Example 2.9]{CSW4}). 
 
 Recall from \cite[Section 6.1]{CSW4} the formal Lie group
\[L\ltimes \CG_{\q}:=(L\times \CG_{\q},m, \varepsilon),\] which is called the \textbf{semidirect product} of $L$ and $\CG_{\q}$. 
The multiplication morphism $m$ is the composition \be\label{eq:mGltimesH}\xymatrix@R-0.7pc@C+2.9pc{m:\ L\times \CG_{\q} \times L\times \CG_\q \ar[r]^{\mathrm{id}_L\times \nu \times \mathrm{id}_{\CG_{\q}}}&L\times L \times \CG_{\q} \times \CG_{\q} \ar[d]_{\mathrm{id}_L\times(\mathrm{id}_L\times i)\times \mathrm{id}_{\CG_{\q}} \times \mathrm{id}_{\CG_{\q}}}\\&L\times L\times L\times \CG_{\q} \times \CG_{\q}\ar[r]^-{\mathrm{id}_L\times \mathrm{id}_{L}\times\psi \times \mathrm{id}_{\CG_{\q}}}&L\times L\times \CG_{\q}\times \CG_{\q}\ar[d]_{m\times m}\\&&L\times \CG_{\q}, }\ee
the unit morphism $\varepsilon$ 
is  
$\Spec(\BC)=\Spec (\BC)\times \Spec(\BC)\xrightarrow{\varepsilon\times \varepsilon} L\times \CG_{\q} $, and the inversion morphism is the composition \be \label{eq:iGH} i:\ L\times \CG_{\q} \xrightarrow {(i\times\mathrm{id}_{L})\times i }L\times L \times \CG_{\q} \xrightarrow{\mathrm{id}_{L}\times \psi} L\times \CG_{\q}.\ee
Here,  $\nu$ is the switching isomorphism, and \be\label{eq:psiofLonLq} \psi=(\overline{\psi},\psi^*):\  L\curvearrowright \CG_{\q}\ee is a formal action 
as automorphisms induced by $\Ad: L\curvearrowright \q$ (see \cite[Proposition 8.1]{CSW4} or \eqref{eq:psii}).

By Proposition \ref{prop:formalliegrouptohopfalg0},  $\CO_{L\ltimes \CG_{\q}}(L\times \{e\})$ is a Hopf formal  algebra, and $\RD^{-\infty}_c(L\times \{e\};\CO_{L\ltimes \CG_{\q}})$ is a Hopf formal coalgebra. 

Let $\RD^{-\infty}_c(L):=(\RC^{\infty}(L))'$ denote the space of (complex-valued) compactly supported distributions on $L$. 
It is well known that 
\[\RD^{-\infty}_c(L)\wt\otimes_{\mathrm{i}}\RU(\q)
      = \RD^{-\infty}_c(L)\otimes_{\mathrm{i}}\RU(\q)\] as LCS 
and that 
\be\label{eq:forcom}(\RD^{-\infty}_c(L)\otimes_{\mathrm{i}}\RU(\q))\wt\otimes_\Ri (\RD^{-\infty}_c(L)\otimes_{\mathrm{i}}\RU(\q))=(\RD^{-\infty}_c(L)\wt\otimes_{\mathrm{i}}\RD^{-\infty}_c(L))\otimes_{\mathrm{i}}(\RU(\q)\otimes_\Ri\RU(\q))\ee as LCS (see \cite[Example B.3]{CSW2}).

\begin{lemd}\label{lem:OLLqpro}
\noindent (a) There are identifications
\be\label{eq:OLLq1}
\begin{aligned}
\CO_{L\times \CG_{\q}}(L\times \{e\})
&=\CO_L(L)\wt\otimes_{\pi} \CO_{\CG_{\q}}(\{e\}) \\
&=\CO_L(L)\wt\otimes_{\pi} (\RU(\q))'
=\Hom_{\BC}(\RU(\q),\RC^{\infty}(L))
\end{aligned}
\ee
between formal algebras. Here, $\Hom_{\BC}(\RU(\q),\RC^{\infty}(L))$ is a formal algebra endowed with the pointwise convergence topology and the canonical multiplication
\[
\begin{array}{rcl}
\Hom_{\BC}(\RU(\q),\RC^{\infty}(L))\times \Hom_{\BC}(\RU(\q),\RC^{\infty}(L))
&\longrightarrow& \Hom_{\BC}(\RU(\q),\RC^{\infty}(L)), \\[2mm]
(f_1,f_2)&\longmapsto& \Bigl(\eta\mapsto \sum_i f_1(\eta_i^{[1]})f_2(\eta_i^{[2]})\Bigr),
\end{array}
\]
where for each $\eta\in \RU(\q)$,
\[
\sum_i\eta_i^{[1]}\otimes\eta_i^{[2]}=\Delta(\eta)\in \RU(\q)\otimes_\Ri\RU(\q),
\]
and $\Delta$ denotes the comultiplication of the Hopf formal coalgebra $\RU(\q)$.

\noindent (b) There are identifications
\be \label{eq:DtoDotimes}
\begin{aligned}
\RD^{-\infty}_c(L\times\{e\};\CO_{L\times \CG_{\q}})
&= (\CO_{L\times \CG_{\q}}(L\times \{e\}))'
= (\CO_L(L)\wt\otimes_{\pi} \RU(\q)')' \\
&= \RD^{-\infty}_c(L)\wt\otimes_{\mathrm{i}}\RU(\q)
= \RD^{-\infty}_c(L)\otimes_{\mathrm{i}}\RU(\q)
\end{aligned}
\ee
between formal coalgebras. Here, $\RD^{-\infty}_c(L)\otimes_{\mathrm{i}}\RU(\q)$ is a formal coalgebra with the canonical comultiplication
\[
\begin{array}{rcl}
\RD^{-\infty}_c(L)\otimes_{\mathrm{i}}\RU(\q)
&\longrightarrow&
(\RD^{-\infty}_c(L)\wt\otimes_{\mathrm{i}}\RD^{-\infty}_c(L))\otimes_{\mathrm{i}}(\RU(\q)\otimes_\Ri\RU(\q))
\quad\text{(see \eqref{eq:forcom})}, \\[2mm]
\tau\otimes\eta&\longmapsto& \Delta(\tau)\otimes\Delta(\eta),
\end{array}
\]
where $\Delta$ denotes the comultiplications of the Hopf formal coalgebras $\RD^{-\infty}_c(L)$ and $\RU(\q)$.
\end{lemd}

\begin{proof}
Assertion (a) follows from \eqref{eq:O_3} and \cite[Corollary B.9]{CSW2}; assertion (b) follows from (a) together with \cite[Theorem 1.3, Lemma A.12 and Example B.3]{CSW2}.
\end{proof}

Note that $\Hom_{\BC}(\RU(\q),\RC^{\infty}(L))$ coincides with the space
\be\label{eq:Hom()}
\left\{
f: L\times \RU(\q) \rightarrow \BC \;\middle|\;
\begin{aligned}
&\text{$f$ is smooth in the first variable,} \\
&\text{and linear in the second variable}
\end{aligned}
\right\}.
\ee
In what follows, we shall often identify $\Hom_{\BC}(\RU(\q),\RC^{\infty}(L))$ with the space in \eqref{eq:Hom()} without further comment.

The natural pairing
\[
\Hom_{\BC}(\RU(\q),\RC^{\infty}(L))\times (\RD^{-\infty}_c(L)\otimes_{\Ri}\RU(\q)) \rightarrow \BC
\]
induced by the duality between $\CO_{L\times\CG_{\q}}(L\times\{e\})$ and $\RD^{-\infty}_c(L\times \{e\};\CO_{L\times \CG_{\q}})$ is then given by
\be \label{eq:pairing1}
(f,\tau\otimes\eta)\mapsto \bigl\langle \tau,\; f(\eta) \bigr\rangle.
\ee

Using \eqref{eq:OLLq1}, it follows from \cite[the proof of Proposition 8.1]{CSW4} that  the homomorphism
\be\label{eq:psii}
\psi^*: \ (\RU(\q))'=\CO_{\CG_\q}(\{e\})
\longrightarrow
\CO_{L\times \CG_{\q}}(L\times\{e\})
=\Hom_{\BC}(\RU(\q),\RC^{\infty}(L))
\ee
is given by
\[
f \longmapsto \bigl((g,\eta)\mapsto f(\Ad_g(\eta))\bigr).
\]
Here
\be\label{eq:AdLonUq}
L\curvearrowright \RU(\q),\qquad (g,\eta)\mapsto \Ad_g(\eta)
\ee
is the representation induced by $\Ad: L\curvearrowright \q$.



As in \eqref{eq:OLLq1}, we  have 
that 
\[ \begin{aligned} \CO_{L\times \CG_{\q}}(L\times\{e\})\wt\otimes_{\pi} \CO_{L\times \CG_{\q}}(L\times \{e\})&=\Hom_{\BC}(\RU(\q),\RC^{\infty}(L))\wt\otimes_{\pi}\Hom_{\BC}(\RU(\q),\RC^{\infty}(L))\\ &=\Hom_{\BC}(\RU(\q)\otimes \RU(\q),\RC^{\infty}(L\times L))\end{aligned} \] as formal algebras. Similar to before, we will also view an element in 
$ \Hom_{\BC}(\RU(\q)\otimes \RU(\q),\RC^{\infty}(L\times L))$ as  a function on $(L\times L)\times (\RU(\q)\otimes \RU(\q))$.  

\begin{prpd}\label{prop:OpfLLq}

\noindent (a)
The formal algebra $\Hom_{\BC}(\RU(\q),\RC^{\infty}(L))$ is a Hopf formal algebra with 
the comultiplication \begin{eqnarray}\label{eq:mulofDLLq}
    \Hom_{\BC}(\RU(\q),\RC^{\infty}(L))&\rightarrow &\Hom_{\BC}(\RU(\q)\otimes \RU(\q),\RC^{\infty}(L\times L)), \nonumber\\ f&\mapsto& \big((g_1,g_2,\eta_1\otimes\eta_2)\mapsto f(g_1g_2, \Ad_{g_2^{-1}}(\eta_1)\eta_2)\big),
\end{eqnarray}
the counit 
\be\label{eq:couniofHom} \Hom_{\BC}(\RU(\q),\RC^{\infty}(L))\rightarrow \BC,\quad f\mapsto f(e,\Ev_e),\ee
and the antipode  
\begin{eqnarray}\label{eq:antiofOLLq}
    \Hom_{\BC}(\RU(\q),\RC^{\infty}(L))&\rightarrow &\Hom_{\BC}(\RU(\q),\RC^{\infty}(L)), \nonumber\\ f&\mapsto& \big((g,\eta)\mapsto f(g^{-1}, \Ad_{g}(\check{\eta}))\big),
\end{eqnarray} 
where 
$\Ev_e$ is the unit element  of the Hopf algebra $\RU(\q)$, 
and $\check{\eta}$ is the image of $\eta$ under the antipode of $\RU(\q)$.
Furthermore, the identification \be\label{eq:OLLqtoHom} \CO_{L\ltimes \CG_{\q}}(L\times \{e\})=\Hom_{\BC}(\RU(\q),\RC^{\infty}(L))\quad \text{(see    \eqref{eq:OLLq1})}\ee  is an identification of Hopf formal algebras. 

\noindent (b) The formal coalgebra $\RD^{-\infty}_c(L)\otimes_{\mathrm{i}}\RU(\q)$ is a Hopf formal coalgebra: the multiplication is the transpose of \eqref{eq:mulofDLLq}, the unit is the transpose of \eqref{eq:couniofHom},  and the antipode is the transpose of \eqref{eq:antiofOLLq}.
Furthermore, the identification 
\be\label{eq:indenDtoDotimes} \RD^{-\infty}_c(L\times \{e\};\CO_{L\ltimes \CG_{\q}})=\RD^{-\infty}_c(L)\otimes_{\mathrm{i}}\RU(\q) \quad \text{(see \eqref{eq:DtoDotimes})}\ee  is an identification of Hopf formal coalgebras. 
\end{prpd}
\begin{proof} 
The assertion (b) follows immediately from the assertion (a). 
    The  Hopf formal algebra structure on $\Hom_{\BC}(\RU(\q),\RC^{\infty}(L))$ stated in the assertion (a) is induced via \eqref{eq:OLLq1} by that  on $\CO_{L\ltimes \CG_{\q}}(L\times \{e\})$, which is easy to verify by using \eqref{eq:psii},  
     the multiplication morphism $m$, the unit morphism $\varepsilon$, and the inversion morphism $i$ of $L\ltimes \CG_{\q}$. 
\end{proof}

 In what follows, for each $\tau\in  \RD^{-\infty}_c(L)$ and $\eta\in \RU(\q)$, let $\tau.\eta$ denote the image of $\tau\otimes\eta$ under the 
homomorphism 
\[  \RD^{-\infty}_c(L)\otimes_{\mathrm{i}}\RU(\q)=\RD^{-\infty}_c(L\times\{e\};\CO_{L\ltimes \CG_{\q}})\xrightarrow{{}^t\psi^*}\RD^{-\infty}_c(\{e\};\CO_{\CG_{\q}})=\RU(\q).\] 

The following result is easy to verify.

\begin{prpd}\label{prop:DofLLq}
The unit of the Hopf formal coalgebra  $\RD^{-\infty}_c(L)\otimes_{\mathrm{i}}\RU(\q)$ is the homomorphism 
\[\BC\rightarrow \RD^{-\infty}_c(L)\otimes_{\mathrm{i}}\RU(\q), \quad c\mapsto c\Ev_e\otimes \Ev_e\quad\text{(the second $\Ev_e$ is the unit element of $\RU(\q)$)}.\] For each $\tau_1\in \RD^{-\infty}_c(L)$, $\eta_1,\eta_2,\eta\in \RU(\q)$, and $\tau\in \RD^{-\infty}_c(L)$ such that the image \[{}^t(\mathrm{id}_L\times i)^*(\tau)\in \RD^{-\infty}_c(L\times L)=\RD^{-\infty}_c(L)\wt\otimes_{\Ri}\RD^{-\infty}_c(L)\] has the form 
$\sum_{i}\tau_i^{[1]}\otimes\tau_i^{[2]}$, the image of $(\tau_1\otimes\eta_1)\otimes(\tau\otimes\eta_2)$ under the multiplication 
\[(\RD^{-\infty}_c(L)\otimes_{\mathrm{i}}\RU(\q))\wt\otimes_{\mathrm{i}}(\RD^{-\infty}_c(L)\otimes_{\mathrm{i}}\RU(\q))\rightarrow \RD^{-\infty}_c(L)\otimes_{\mathrm{i}}\RU(\q)\]
is
\be\label{eq:mulofDDLLq}\sum_i\tau_1\tau^{[1]}_{i}\otimes(\tau_{i}^{[2]}.\eta_1)\eta_2,
\ee
and the image of $\tau\otimes\eta$ under the antipode 
\[\RD^{-\infty}_c(L)\otimes_{\mathrm{i}}\RU(\q)\rightarrow \RD^{-\infty}_c(L)\otimes_{\mathrm{i}}\RU(\q)\]
is 
\be\label{eq:antiofDLLq} \sum_i \tau^{[2]}_i\otimes\tau_i^{[1]}.\check{\eta}. \ee
\end{prpd}

By \eqref{eq:mulofDDLLq}, the following corollary is straightforward.
\begin{cord}
For each $g\in L$ and each \[\eta_0\in \RU(\q)=\{\Ev_e\otimes\eta:\eta\in \RU(\q)\}\subset \RD^{-\infty}_c(L)\otimes_{\mathrm{i}}\RU(\q)= \RD^{-\infty}_c(L\times \{e\};\CO_{L\ltimes \CG_{\q}}) ,\] we have that \be\label{eq:realAd} \Ev_{(g,e)}\eta_0 \Ev_{(g^{-1},e)}=\Ad_{g}(\eta_0)\in \RU(\q) \ee in the Hopf formal coalgebra $\RD^{-\infty}_c(L\times \{e\};\CO_{L\ltimes \CG_{\q}})$.
\end{cord}

\subsection{Topological Hopf algebras associated to $(L\ltimes \CG_{\q})/\CG_{\l}$}
Let $(\q,L,\Ad,\iota)$ be a Lie pair. 
Let $\l$ denote the complexified Lie algebra of $L$.
Recall from \cite[Proposition 8.3]{CSW4} that the composition 
\be\label{eq:idtimesi}  \vartheta: \ \CG_{\l}\xrightarrow{\mathrm{id}_{\CG_{\l}}\times i}\CG_{\l} \times \CG_{\l}\xrightarrow{\iota\times \iota} L\ltimes \CG_{\q} \ee is  a homomorphism between formal Lie groups, where
$\iota: \CG_{\l}\rightarrow L$ is the homomorphism of formal Lie groups determined by the inclusion $\RU(\l)\rightarrow \RD^{-\infty}_c(L)$ and $\iota: \CG_{\l}\rightarrow \CG_{\q}$ is the homomorphism of formal Lie groups determined by the homomorphism $ \RU(\l)\rightarrow \RU(\q)$ (see Proposition \ref{prop:equiMO}), which is induced by $\iota: \l\rightarrow \q$. 
The composition 
\be\label{eq:Llonsemi} (L\ltimes \CG_\q)\times \CG_{\l} \xrightarrow{\mathrm{id}_{L\ltimes \CG_\q}\times \vartheta} (L\ltimes \CG_\q) \times (L\ltimes \CG_\q) \xrightarrow {m} L\ltimes \CG_\q\ee is a right formal action of $\CG_{\l}$ on $L\ltimes \CG_\q$ (see \cite[Proposition 8.3]{CSW4}). 

Recall from \cite[Section 7.1]{CSW4} the definition of the quotient
with respect to a formal action. The following result follows from \cite[Theorem 7.8, Propositions 8.3 and 9.4]{CSW4}.

\begin{prpd}\label{prop:quo}
\noindent (a) The pair
\[
  \bigg(
    (L\ltimes \CG_{\q})/\CG_{\l}
    :=
    \bigl(
      L,\,
      \CO_{(L\ltimes \CG_{\q})/\CG_{\l}}
    \bigr),\,
    \pi = (\overline{\pi},\pi^*)
  \bigg)
\]
is the quotient of $L\ltimes \CG_\q$ by $\CG_{\l}$ with respect to the right formal action \eqref{eq:Llonsemi}.
Here, the structure sheaf $\CO_{(L\ltimes \CG_{\q})/\CG_{\l}}$ is given by
\[
  V \mapsto
  \CO_{(L\ltimes \CG_{\q})/\CG_{\l}}(V)
  :=
  \CO_{L\ltimes \CG_{\q}}(V\times\{e\})^{\CG_{\l}}
  \quad\text{($V$ an open subset of $L$)},
\]
the continuous map
\[
  L = L\times\{e\} \xrightarrow{\overline{\pi}} L
\]
coincides with $\mathrm{id}_{L}$, and the homomorphism $\pi^*$ is given by the inclusion
\[
  \CO_{(L\ltimes \CG_{\q})/\CG_{\l}}(L)
  =
  \CO_{L\ltimes \CG_{\q}}(L\times \{e\})^{\CG_{\l}}
  \rightarrow
  \CO_{L\ltimes \CG_{\q}}(L\times\{e\}),
\]
where $\CO_{L\ltimes \CG_{\q}}(V\times\{e\})^{\CG_{\l}}$ is the subspace of $\CO_{L\ltimes \CG_{\q}}(V\times\{e\})$ consisting of all $\CG_{\l}$-invariant formal functions (see \cite[Definition 7.2]{CSW4}).

\noindent (b) The homomorphism
\[
  \pi^*:
  \CO_{(L\ltimes \CG_{\q})/\CG_{\l}}(L)
  \rightarrow
  \CO_{L\ltimes \CG_{\q}}(L\times\{e\})
\]
is a closed embedding. In particular, equipped with the subspace topology of $\CO_{L\ltimes \CG_{\q}}(L\times\{e\})$, the LCS $\CO_{L\ltimes \CG_{\q}}(L\times\{e\})^{\CG_{\l}}$ is isomorphic to the LCS $\CO_{(L\ltimes \CG_{\q})/\CG_{\l}}(L)$.

\noindent (c) There is a unique morphism
\[
  m:
  \bigl((L\ltimes \CG_{\q})/\CG_{\l}\bigr)
  \times
  \bigl((L\ltimes \CG_{\q})/\CG_{\l}\bigr)
  \rightarrow
  (L\ltimes \CG_{\q})/\CG_{\l}
\]
such that the diagram
\be \label{eq:diaofm0}
\begin{CD}
    (L\ltimes \CG_{\q})\times (L\ltimes \CG_{\q}) @>m>> L\ltimes \CG_{\q} \\
    @V\pi \times \pi VV @VV\pi V\\
    ((L\ltimes \CG_{\q})/\CG_{\l})\times ((L\ltimes \CG_{\q})/\CG_{\l}) @>m>> (L\ltimes \CG_{\q})/\CG_{\l}
\end{CD}
\ee
commutes. Furthermore, together with the multiplication morphism $m$ and the unit morphism
\[
  \varepsilon:
  \Spec(\BC) \xrightarrow{\varepsilon} L\ltimes \CG_{\q} \xrightarrow{\pi} (L\ltimes \CG_{\q})/\CG_{\l},
\]
the formal manifold $(L\ltimes \CG_{\q})/\CG_{\l}$ becomes a formal Lie group. In addition, the quotient morphism $\pi: L\ltimes \CG_{\q}\to (L\ltimes \CG_{\q})/\CG_{\l}$ is a homomorphism of formal Lie groups.

\noindent (d) The functors \eqref{eq:Gtopair} and
\be \label{eq:funcpairtogroup}
  (\q,L) \mapsto (L\ltimes \mathcal{G}_{\q})/\mathcal{G}_{\l}
\ee
are quasi-inverse of each other.
\end{prpd}

By Propositions \ref{prop:formalliegrouptohopfalg0} and \ref{prop:quo}, we have that 
$\CO_{(L\ltimes \CG_{\q})/\CG_{\l}}(L)$ is a  Hopf formal  algebra, and that 
$\RD^{-\infty}_c(L;\CO_{(L\ltimes \CG_{\q})/\CG_{\l}})$ is a Hopf formal coalgebra.

Recall from Proposition \ref{prop:OpfLLq} that $\Hom_{\BC}(\RU(\q),\RC^{\infty}(L))$ is a  Hopf formal algebra, and that $\RD^{-\infty}_c(L)\otimes_{\Ri}\RU(\q)$ is a Hopf formal coalgebra. 

Let $\Hom_{\RU(\l)}(\RU(\q),\RC^{\infty}(L))$ denote the subspace  of $\Hom_{\BC}(\RU(\q),\RC^{\infty}(L))$ consisting of all $\RU(\l)$-equivariant maps, where $\RU(\l)$ acts on $\RU(\q)$ by left multiplication, and on $\RC^{\infty}(L)$ via
\[\RU(\l)\times \RC^{\infty}(L)\rightarrow \RC^{\infty}(L),\quad
(\tau,f)\mapsto(g\mapsto \la \Ev_g\tau,f\ra
).
\] Equip $\Hom_{\RU(\l)}(\RU(\q),\RC^{\infty}(L))$ with the subspace topology of $\Hom_{\BC}(\RU(\q),\RC^{\infty}(L))$. Then it becomes an LCS.

The tensor product $\RD^{-\infty}_c(L)\otimes_{\RU(\l)}\RU(\q)$ is a quotient space of $\RD^{-\infty}_c(L)\otimes_{\Ri}\RU(\q)$. Equipped with the quotient topology, it becomes an LCS.


\begin{prpd}\label{prop:OofLLqLl}
\noindent (a) The LCS $\Hom_{\RU(\l)}(\RU(\q),\RC^{\infty}(L))$ is a Hopf formal algebra, whose Hopf formal algebra structure is the unique one such that the inclusion map 
\be\label{eq:HomultoHom}
\Hom_{\RU(\l)}(\RU(\q),\RC^{\infty}(L)) \rightarrow \Hom_{\BC}(\RU(\q),\RC^{\infty}(L))
\ee 
is a homomorphism between Hopf formal algebras. Moreover, there is a unique identification 
\be\label{eq:=onOG}
\CO_{(L\ltimes \CG_{\q})/\CG_{\l}}(L)
   =\Hom_{\RU(\l)}(\RU(\q),\RC^{\infty}(L))
\ee 
between Hopf formal algebras such that the diagram 
\[
\begin{CD}
    \CO_{(L\ltimes \CG_{\q})/\CG_{\l}}(L) @>\eqref{eq:=onOG}>> \Hom_{\RU(\l)}(\RU(\q),\RC^{\infty}(L)) \\
    @V\pi^*VV @VV\eqref{eq:HomultoHom}V \\
    \CO_{L\ltimes \CG_{\q}}(L\times\{e\}) @>\eqref{eq:OLLqtoHom}>> \Hom_{\BC}(\RU(\q),\RC^{\infty}(L))
\end{CD}
\]
commutes.

\noindent (b) The LCS $\RD^{-\infty}_c(L)\otimes_{\RU(\l)}\RU(\q)$ is a Hopf formal coalgebra, whose Hopf formal coalgebra structure is the unique one such that the quotient map
\be \label{eq:DtoD_U}
\RD^{-\infty}_c(L)\otimes_{\Ri}\RU(\q)
   \longrightarrow
   \RD^{-\infty}_c(L)\otimes_{\RU(\l)}\RU(\q)
\ee
is a homomorphism between Hopf formal coalgebras. Moreover, there is a unique identification
\be\label{eq:D=DU}
\RD^{-\infty}_c(L;\CO_{(L\ltimes \CG_{\q})/\CG_{\l}})
   =\RD^{-\infty}_c(L)\otimes_{\RU(\l)}\RU(\q)
\ee
between Hopf formal coalgebras such that the diagram 
\[
\begin{CD}
    \RD^{-\infty}_c(L\times\{e\};\CO_{L\ltimes \CG_{\q}}) @>\eqref{eq:indenDtoDotimes}>> \RD^{-\infty}_c(L)\otimes_{\Ri}\RU(\q) \\
    @V{}^t\pi^*VV @VV\eqref{eq:DtoD_U}V \\
    \RD^{-\infty}_c(L;\CO_{(L\ltimes \CG_{\q})/\CG_{\l}}) @>\eqref{eq:D=DU}>> \RD^{-\infty}_c(L)\otimes_{\RU(\l)}\RU(\q)
\end{CD}
\]
commutes.

\noindent (c) The pairing  
\[
\Hom_{\RU(\l)}(\RU(\q),\RC^{\infty}(L)) \times (\RD^{-\infty}_c(L)\otimes_{\RU(\l)}\RU(\q)) \rightarrow \BC
\]
induced by the pairing between $\CO_{(L\ltimes\CG_{\q})/\CG_{\l}}(L)$ and $\RD^{-\infty}_c(L;\CO_{(L\ltimes \CG_{\q})/\CG_{\l}})$ is given by
\be \label{eq:pairing2}
(f,\tau\otimes\eta) \mapsto \bigl\langle \tau,  f(\eta)  \bigr\rangle.
\ee
\end{prpd}


The following result is a direct consequence of \eqref{eq:realAd} together with Propositions \ref{prop:quo} and \ref{prop:OofLLqLl}.

\begin{cord}\label{cor:realAd}
  Let $G$ be a formal Lie group. Then the adjoint action $\Ad: \underline{G}\curvearrowright\g$ coincides with the action 
\be\underline{G}\curvearrowright \g, \quad (g,\eta)\mapsto \Ev_g\eta\Ev_{g^{-1}}\quad\text{(see \eqref{eq:eta12})}.\ee
\end{cord}

In the rest of this subsection, we prove Proposition~\ref{prop:OofLLqLl}. 

As in \eqref{eq:Hom()}, an element $f\in\Hom_{\RU(\l)}(\RU(\q),\RC^\infty(L))$ may be regarded as a map
\be \label{eq:Ueq}
f:\ L\times\RU(\q)\longrightarrow \BC
\ee
that is smooth in the first variable, linear in the second variable, and satisfies
\[
f(g,\tau\eta)=\bigl\langle \Ev_g\tau,\; f(\,\cdot\,,\eta) \bigr\rangle
\]
for all $\tau\in\RU(\l)$, $g\in L$, and $\eta\in\RU(\q)$. Here, $f(\,\cdot\,,\eta)$ is the smooth function $x\mapsto f(x,\eta)$ on $L$.

\begin{lemd}\label{lem:OtoHoml}
Under the identification $\CO_{L\ltimes \CG_{\q}}(L\times \{e\})=\Hom_{\BC}(\RU(\q), \RC^{\infty}(L))$ (see \eqref{eq:OLLqtoHom}), we have that
\be\label{eq:OsubHom=}
\CO_{L\ltimes \CG_{\q}}(L\times \{e\})^{\CG_{\l}}
=
\Hom_{\RU(\l)}(\RU(\q), \RC^{\infty}(L))
\ee
as LCS.
\end{lemd}

\begin{proof}
By the commutative diagram \eqref{eq:chainofG} and the fact that
\[
\vartheta=(\overline{\vartheta},\vartheta^*):\ \CG_{\l}\rightarrow L\ltimes \CG_{\q}
\]
is a homomorphism between formal Lie groups, it is straightforward that
\[
\RU(\l)\curvearrowright \CO_{L\ltimes \CG_{\q}}(L\times \{e\}),\qquad 
(\tau,f)\mapsto \tau.f
\]
is a continuous $\RU(\l)$-module structure, where $\tau.f$ is the image of $f$ under the composition
\be \label{eq:OtoU}
\begin{aligned}
&\CO_{L\ltimes \CG_{\q}}(L\times \{e\})
\xrightarrow{m^*} 
\CO_{L\ltimes \CG_{\q}}(L\times \{e\}) \wt\otimes_{\pi} \CO_{L\ltimes \CG_{\q}}(L\times \{e\})  
\xrightarrow{\mathrm{id}_{\CO_{L\ltimes \CG_{\q}}(L\times \{e\})}\otimes \vartheta^*} \\
&\qquad \CO_{L\ltimes \CG_{\q}}(L\times \{e\}) \wt\otimes_{\pi} (\RU(\l))' 
\xrightarrow{\mathrm{id}_{\CO_{L\ltimes \CG_{\q}}(L\times \{e\})}\otimes \tau}
\CO_{L\ltimes \CG_{\q}}(L\times \{e\}).
\end{aligned}
\ee
Then for each $f\in \CO_{L\ltimes \CG_{\q}}(L\times \{e\})$, one has that
\[
f\in \CO_{L\ltimes \CG_{\q}}(L\times \{e\})^{\CG_{\l}}
\quad\Longleftrightarrow\quad
\tau.f=0\quad\text{for all }\tau\in \l.
\]

Using \eqref{eq:OtoU} and \eqref{eq:mulofDLLq}, it is straightforward that, under the identification $\CO_{L\ltimes \CG_{\q}}(L\times \{e\})=\Hom_{\BC}(\RU(\q),\RC^{\infty}(L))$,
\be\label{eq:tauonf}
(\tau.f)(g,\eta)=
\bigl\langle \Ev_g\tau,\; f(\,\cdot\,,\eta) \bigr\rangle - f(g,\tau\eta)
\ee
for all $\tau\in\l$, $f\in \CO_{L\ltimes \CG_{\q}}(L\times \{e\})$, $g\in L$, and $\eta\in \RU(\q)$.

Note that $f\in \Hom_{\RU(\l)}(\RU(\q), \RC^{\infty}(L))$ if and only if $f\in \Hom_{\l}(\RU(\q), \RC^{\infty}(L))$. Then the lemma follows immediately from \eqref{eq:Ueq}.
\end{proof}


By taking the transpose of the quotient map \eqref{eq:DtoD_U}, we obtain an injective continuous $\BC$-linear map 
\be\label{eq:D'toHom} (\RD^{-\infty}_c(L)\otimes_{\RU(\l)}\RU(\q))'\rightarrow \Hom_{\BC}(\RU(\q),\RC^{\infty}(L)),\  f\mapsto \big((g,\eta)\mapsto f(\Ev_g\otimes\eta)\big).\ee

\begin{lemd}\label{lem:D'toHom} The image of \eqref{eq:D'toHom} is contained in $\Hom_{\RU(\l)}(\RU(\q),\RC^{\infty}(L))$. Furthermore, the 
continuous map 
    \be \label{eq:D'toHoml}(\RD^{-\infty}_c(L)\otimes_{\RU(\l)}\RU(\q))'\rightarrow \Hom_{\RU(\l)}(\RU(\q),\RC^{\infty}(L)), \ f\mapsto ((g,\eta)\mapsto f(\Ev_g\otimes\eta))\ee induced by \eqref{eq:D'toHom}  is an isomorphism of LCS.
\end{lemd}
\begin{proof}
  For each $f\in(\RD^{-\infty}_c(L)\otimes_{\RU(\l)}\RU(\q))' $ and  $\tau\in \RU(\l)$, it is clear that 
    \[f(\Ev_g\tau\otimes\eta)=f(\Ev_g\otimes\tau\eta),\quad \text{for all $g\in L$ and $\eta\in \RU(\q)$}.\]
    This implies that the image of $f$ under the map \eqref{eq:D'toHom} is in $\Hom_{\RU(\l)}(\RU(\q),\RC^{\infty}(L))$. Then the first assertion follows. 

Now we prove the second assertion.
Let $\l^c$ be a $\BC$-linear subspace of $\q$ such that $\l\oplus\l^c=\q$. 
It is well-known that there is an isomorphism 
\be\label{eq:US=U} \begin{array}{rcl}
    \RU(\l)\otimes_{\Ri} \RS(\l^c) &  \rightarrow& \RU(\q)\\
    \tau\otimes\eta &   \mapsto& \tau\mathrm{symm}(\eta)
\end{array}\ee
of LCS (\cf \cite[0.4.3]{Wa}). Here $\RS(\l^c)$
is the symmetric algebra of $\l^c$, equipped with the finest
locally convex topology (see \cite[Example B.3]{CSW2}), and \[\mathrm{symm}: \ \RS(\l^c)\rightarrow \RU(\q)\] is the linear map such that 
\[\mathrm{symm}(X_1X_2\cdots X_k)=(1/k!)\sum_{\sigma\in \RS_k} X_{\sigma(1)}X_{\sigma(2)}\cdots X_{\sigma(k)}\]
for each $k\geq0$ and $X_1,X_2,\dots,X_k \in \l^{c}$, where $\RS_k$ is the symmetric group
of degree $k$. 

It follows from \eqref{eq:US=U} that 
\be\label{eq:D=[]}\RD^{-\infty}_c(L)\otimes_{\Ri} \RS(\l^c)\cong \RD^{-\infty}_c(L)\otimes_{\RU(\l)}\RU(\q)\ee
as LCS, and that 
\[\Hom_{\RU(\l)}(\RU(\q),\RC^{\infty}(L))\cong \Hom_{\BC}(\RS(\l^c),\RC^{\infty}(L))\] as LCS.
Moreover, 
\[(\RD^{-\infty}_c(L)\otimes_{\Ri} \RS(\l^c))'\cong \Hom_{\BC}(\RS(\l^c),\RC^{\infty}(L))\]
as LCS.
 Under these identifications, the map \eqref{eq:D'toHoml} coincides with the identity map 
$\mathrm{id}_{\Hom_{\BC}(\RS(\l^c),\RC^{\infty}(L))}$. 
Therefore, \eqref{eq:D'toHoml} is an isomorphism of LCS.
 \end{proof}
\noindent\textbf{Proof of Proposition \ref{prop:OofLLqLl}:} The assertion (a) in Proposition \ref{prop:OofLLqLl} follows from Lemma \ref{lem:OtoHoml}. 
Using the assertion (a) and Lemma \ref{lem:D'toHom}, we have that \[(\RD^{-\infty}_c(L)\otimes_{\RU(\l)}\RU(\q))'= \Hom_{\RU(\l)}(\RU(\q),\RC^{\infty}(L))=\CO_{(L\ltimes \CG_{\q})/\CG_{\l}}(L)\] is a reflexive LCS. On the other hand, it follows from \eqref{eq:D=[]}  that 
$\RD^{-\infty}_c(L)\otimes_{\RU(\l)}\RU(\q)$ is a reflexive LCS. 
Then there is an identification 
\[\RD^{-\infty}_c(L)\otimes_{\RU(\l)}\RU(\q)=\RD^{-\infty}_c(L;\CO_{(L\ltimes \CG_{\q})/\CG_{\l}})\]
and the assertion (b) follows immediately. The assertion (c) follows from \eqref{eq:pairing1}. 
\qed

\section{Modules of Lie pairs}\label{sec:repofGG}
In this section, we construct some modules of Lie pairs on various function spaces over formal manifolds introduced in \cite{CSW1} and \cite{CSW2}.

\subsection{Function spaces over formal manifolds}

Let $M$ be a formal manifold. 
In \cite{CSW1,CSW2}, various function spaces over $M$ are studied, including 
the space $\CO_M(M)$ of formal functions, 
the space $\RD^{-\infty}_c(M;\CO_M)$ of compactly supported formal distributions, 
the space $\RD^{\infty}_c(M;\CO_M)$ of compactly supported formal densities, 
and the space $\RC^{-\infty}(M;\CO_M)$ of formal generalized functions.
Equipped with their canonical topologies, these spaces are complete reflexive LCS 
(see \cite[Lemma A.10]{CSW2}). 
By \cite[Example 4.2]{CSW1} and \cite[Theorem 1.1]{CSW2}, these spaces generalize, respectively, 
the spaces of (complex-valued) smooth functions, compactly supported distributions, 
compactly supported smooth densities, and generalized functions on a smooth manifold.

For the spaces $\CO_M(M)$ and $\RD^{-\infty}_c(M;\CO_M)$, we refer the reader to Section~\ref{sec:comodofHopf}. 
The space
\[
\RC^{-\infty}(M;\CO_M) := \bigl(\RD^{\infty}_c(M;\CO_M)\bigr)'
\]
is equipped with the strong dual topology (see \cite[(3.1)]{CSW2}), 
where the LCS $\RD^{\infty}_c(M;\CO_M)$ is recalled in the remainder of this subsection.

Recall from  \cite[(2.3)]{CSW2} that as a vector space 
\[
\RD^{\infty}_c(M;\CO_M)
:= \rho_M\bigl(\mathrm{Diff}_{c}(\CO_M,\underline{\mathcal{D}_M})\bigr)
\subset \bigl(\CO_M(M)\bigr)' = \RD^{-\infty}_c(M;\CO_M).
\]
 Here $\rho_M$ is the linear map
\begin{equation}
\begin{split}
\rho_M:\ \mathrm{Diff}_{c}(\CO_M,\underline{\mathcal{D}_M})
&\longrightarrow \mathrm{Hom}_{\mathbb C}(\CO_M(M),\mathbb C),\\
D &\longmapsto \Bigl( f \mapsto \int_{M} D(f) \Bigr),
\end{split}
\label{eq:defcomsuppden}
\end{equation}
$\underline{\mathcal{D}_M}$ is the sheaf of all (complex-valued) smooth densities on the reduction 
$\underline{M}$, which is a $\CO_M$-module via the quotient homomorphism $\CO_M\to\underline{\CO_M}$, 
and $\mathrm{Diff}_{c}(\CO_M,\underline{\mathcal{D}_M})$ is the space of all compactly supported differential operators from $\CO_M$ 
to $\underline{\mathcal{D}_M}$ (see \cite[Section 3.2]{CSW1}).


We next recall the inductive limit topology on $\RD^{\infty}_c(M;\CO_M)$.
Equip the space $\mathrm{Diff}_{\mathrm{fin}} (\CO_M,\underline{\mathcal{D}_M})$ of  finite-order differential operators (see \cite[Definition 3.6]{CSW1}) with the topology given by the seminorms \[\{\abs{\,\cdot\,}_{X,f}\}_{X\in \mathrm{Diff}_c(\underline{\mathcal{D}_M},\underline{\CO_M}), f\in \CO_M(M)},\] where \[\abs{D}_{X,f}=\sup_{a\in M}|\big((X\circ D)(f)\big)(a)| \quad\text{for all $D\in \mathrm{Diff}_{\mathrm{fin}} (\CO_M,\underline{\mathcal{D}_M})$}.\] Then $\mathrm{Diff}_{\mathrm{fin}} (\CO_M,\underline{\mathcal{D}_M})$ becomes an LCS. Equip $\underline{\mathcal{D}_M}(M)$ with the smooth topology (see \cite[Definition 4.1]{CSW1}). Then it follows from \cite[Lemma 2.8]{CSW2}  that 
\be\label{eq:convergence}
\begin{aligned}
&\text{a net $\{D_i\}_{i\in I}$  in $\mathrm{Diff}_{\mathrm{fin}}(\CO_M,\underline{\mathcal{D}_M})$
 converges to $0$ if and only if }  \\ & \text{$\{D_i(f)\}_{i\in I}$ converges to $0$ in $\underline{\mathcal{D}_M}(M)$ for all $f\in \CO_M(M)$.}
\end{aligned}
\ee

There is a  directed set $(\mathcal{C}(M),\preceq)$, where
\be\label{eq:C(M)} \mathcal{C}(M):=\{(K,r)\mid \text{$K$ is a compact subset of $M$ and $r\in \BN$}\},\ee
and for two pairs $(K,r),(K',r')\in \mathcal{C}(M)$, 
\[\text{$(K,r)\preceq(K',r')$ if and only if $K\subset K'$ and $r\le r'$.}\] Equip \be \label{eq:D_ctop} \mathrm{D}^\infty_c(M;\CO_M)=\varinjlim_{(K,r)\in \mathcal{C}(M)}\mathrm{D}^\infty_{K,r}(M;\CO_M)\ee with the inductive limit topology. 
Here for each $(K,r)\in \mathcal{C}(M)$, \[ \mathrm{D}^\infty_{K,r}(M;\CO_M) \mid =\rho_M\bigl(\mathrm{Diff}_{K,r}(\CO_M,\underline{\mathcal{D}_M})\bigr), \]
and $\mathrm{Diff}_{K,r}(\CO_M,\underline{\mathcal{D}_M})$ denotes the subspace of $\mathrm{Diff}_{\mathrm{fin}}(\CO_M,\underline{\mathcal{D}_M})$ consisting of all  differential operators supported in $K$ and of order at most $r$.
The space $ \mathrm{Diff}_{K,r}(\CO_M,\underline{  \mathcal{D}_M}) $ is endowed with the subspace topology of $\mathrm{Diff}_{\mathrm{fin}}(\CO_M,\underline{\mathcal{D}_M})$, and the space $ \mathrm{D}^\infty_{K,r}(M;\CO_M) $ is equipped with the quotient topology of $\mathrm{Diff}_{K,r}(\CO_M,\underline{\mathcal{D}_M})$.

\subsection{Dual modules}\label{sec:dualrep} 
Recall that an LCS $E$ is called semi-reflexive if the canonical map $E\rightarrow (E')'$ is bijective.  
\begin{prpd}\label{prop:contraofgG}
   Let $(\q,L)$ be a Lie pair, and let $(\mu,E)$ be a $(\q,L)$-module.
Assume that $E$ is semi-reflexive and that $E'$ is quasi-complete. Then $E'$ together with the contragredient actions 
\be\label{eq:conofG} L \curvearrowright E', \quad (g,v)\mapsto g.v:=(u\mapsto \la v,g^{-1}.u\ra)\quad (u\in E),
\ee and 
\be\label{eq:conofg}  \q \curvearrowright E', \quad (\eta,v)\mapsto \eta.v:=(u\mapsto \la v,-\eta.u\ra)\ee
is a $(\q, L)$-module.
\end{prpd}
\begin{proof}
     The continuity of \eqref{eq:conofG} follows from the semi-reflexivity of $E$ (\cf \cite[Corollary 4.1.2.3]{Ga}). Since $\q$ is finite-dimensional, for every LCS $E_0$,
\be\label{eq:algconti}
\text{every separately continuous bilinear map $\q\times E_0\rightarrow E_0$ is continuous.}
\ee
Thus, the continuity of \eqref{eq:conofg} follows from its separate continuity. It remains to verify the two compatibility conditions
\eqref{eq:qLmodule1} and \eqref{eq:qLmodule2}.

Let $v\in E'$. For the first compatibility condition, it is clear  that \[ \la  g.(\eta.(g^{-1}.v)),u\ra =\la v, g.((-\eta).(g^{-1}.u))\ra =\la v, (-\Ad_g\eta).u\ra= \la(\Ad_g\eta).v,u\ra\] for all $g\in L$, $\eta\in\q$, and $u\in E$. 
Hence, 
$
g.(\eta.(g^{-1}.v))=(\Ad_g\eta).v$.

For the second one, \cite[Theorem 4.5]{Ca} implies that the continuous action
$L\curvearrowright E'$ is smooth.  Hence, for every $\tau\in\Lie(L)$, the limit
\[
\lim_{t\to 0}\frac{\exp(t\tau).v-v}{t}
\]
exists in $E'$. Note that for every $u\in E$, \begin{eqnarray*} \la\lim_{t\rightarrow 0}\frac{\exp(t\tau).v-v}{t},u\ra&=&\lim_{t\rightarrow 0}\la \frac{\exp(t\tau).v-v}{t},u \ra=\la v,\lim_{t\rightarrow 0}\frac{\exp(-t\tau).u-u}{t} \ra\\ &=&\la v,\iota(-\tau).u\ra=\la\iota(\tau).v,u\ra. \end{eqnarray*} Therefore, \[\lim_{t\to 0}\frac{\exp(t\tau).v-v}{t}=\iota(\tau).v\] and the proposition follows. \end{proof}

\subsection{Modules of $(\g,\underline{G})$}


We first fix some notation that will be used frequently throughout this paper.

Let $E_1$ and $E_2$ be two quasi-complete LCS. For  $\eta_1\in E_1'$ and $\eta_2\in E'_2$, let 
\be\label{eq:ra12} \la\eta_1, u\ra_1\qaq\la\eta_2,u\ra_2\ee denote the images of $u\in E_1\wt\otimes_{\pi} E_2$ under the continuous maps
\[E_1\wt\otimes_{\pi} E_2\xrightarrow{\eta_1\otimes \mathrm{id}_{E_2}}E_2\qaq E_1\wt\otimes_{\pi} E_2\xrightarrow{\mathrm{id}_{E_1}\otimes\eta_2 }E_1,\] respectively. 
Let $E_3$ be another quasi-complete LCS.  For every $\eta_3\in E_3'$, let
\be\label{eq:ra123} \la\eta_1\otimes\eta_2,v\ra_{1,2}\qaq\la\eta_2\otimes\eta_3,v\ra_{2,3} \ee denote 
the images of $v\in E_1\wt\otimes_{\pi} E_2\wt\otimes_{\pi} E_3$ under the continuous maps
\[E_1\wt\otimes_{\pi} E_2\wt\otimes_{\pi} E_3\xrightarrow{\eta_1\otimes\eta_2\otimes \mathrm{id}_{E_3}}E_3 \qaq E_1\wt\otimes_{\pi} E_2\wt\otimes_{\pi} E_3\xrightarrow{\mathrm{id}_{E_1}\otimes\eta_2\otimes\eta_3}E_1,\]
respectively. 


Let $G$ be a formal Lie group. Recall from Proposition \ref{prop:Gtopair} that $(\g,\underline{G})$ is a Lie pair. Let  $\psi=(\overline{\psi},\psi^*): \ G\times M\rightarrow M$ be a formal action of $G$ on a formal manifold $M$ (see \cite[Definition 4.1]{CSW4}).
Then 
\[\psi^*: \ \CO_M(M)\rightarrow \CO_{G\times M}(G\times M)=\CO_G(G)\wt\otimes_{\pi}\CO_M(M) \quad\text{(see \eqref{eq:O_3})}\] is a homomorphism  between formal algebras. 
\begin{prpd}\label{prop:modofgG} 

 \noindent(a) The map   \be\label{eq:LieGonOM} \g\times\CO_M(M)\rightarrow \CO_M(M), \quad (\eta, f)\mapsto \eta.f :=\la -\eta , \psi^*f\ra_1\ee is a continuous Lie algebra action,
	and the map \be\label{eq:GonOM}  \underline G\times\CO_M(M)\rightarrow \CO_M(M), \quad (g,f)\mapsto g.f:=\la \Ev_{g^{-1}}, \psi^*f\ra_1\ee is a continuous group action. Furthermore, together with \eqref{eq:LieGonOM} and \eqref{eq:GonOM}, the LCS $\CO_M(M)$ is a $(\g, \underline G)$-module.
    
 \noindent(b) The LCS $\RD^{-\infty}_c(M;\CO_M)$, together with 
 the contragredient actions of \eqref{eq:LieGonOM} and \eqref{eq:GonOM}, is a $(\g, \underline G)$-module. 
 
 \noindent(c) The subspace $\RD^{\infty}_c(M;\CO_M)$ of $\RD^{-\infty}_c(M;\CO_M)$ is $(\g, \underline G)$-stable. Furthermore, together with the restrictions of the contragredient actions
in (b),  the LCS $\RD^{\infty}_c(M;\CO_M)$ becomes a $(\g, \underline G)$-module.

  \noindent(d)
    The LCS $\RC^{-\infty}(M;\CO_M)$, together with the actions
contragredient to those on $\RD^{\infty}_c(M;\CO_M)$ in (c),  
    is a $(\g, \underline G)$-module. 

\end{prpd}

The remainder of this subsection is devoted to the proof of Proposition \ref{prop:modofgG}. 
In view of Proposition \ref{prop:contraofgG}, assertions (b) and (d) follow respectively from assertions (a) and (c). 
It remains to prove assertions (a) and (c).

\begin{lemd}\label{lem:GcontinuonOM}
    The maps \eqref{eq:LieGonOM} and \eqref{eq:GonOM} are continuous. 
\end{lemd}
\begin{proof}
 Using \eqref{eq:algconti}, it is clear that the map \eqref{eq:LieGonOM} is continuous.   Note that  the map \eqref{eq:GonOM} is the composition of the continuous map
    \[G\times \CO_M(M)\xrightarrow{\mathrm{id}_{G} \times \psi^*} G\times  (\CO_G(G)\wt\otimes_{\pi} \CO_M(M))\]
    and the map 
    \be \label{eq:wttimes2}G\times (\CO_G(G)\wt\otimes_{\pi}\CO_M(M) )\rightarrow \CO_M(M),\quad (g, f\otimes u)\mapsto f(g^{-1})u.\ee 
     Therefore, it suffices to prove that \eqref{eq:wttimes2} is continuous.
    
 Let $U$ be a relatively compact open subset of $G$ and let $K$ be the closure of $U$ in $G$. Then for each $\varepsilon>0$, we have that   
     \[|f(g)|<\varepsilon,\quad  \text{for each $g\in U$ and each $f\in \CO_G(G)$ with $\sup_{g\in K}|f(g)|<\varepsilon$}.\] 
This implies that the map \be\label{eq:GO(G)} G\times \CO_G(G)\rightarrow \BC,\quad  (g,f)\mapsto f(g)\ee  is continuous. Then it is straightforward that the map \eqref{eq:wttimes2} is continuous. This finishes the proof.
\end{proof}
 

\noindent\textbf{Proof of Proposition \ref{prop:modofgG} (a):} 
By the definition of formal actions, it is clear that the continuous map \eqref{eq:LieGonOM} is a Lie algebra action and the continuous map \eqref{eq:GonOM} is a Lie group action. It suffices to prove the second assertion in (a).

 For every  $f\in \CO_G(G)$,  it is clear that
\[
\begin{aligned}
\la \Ad_g(\eta),f\ra
&=\la \Ev_g\eta\Ev_{g^{-1}},f\ra \quad \text{(see Corollary \ref{cor:realAd})}\\
&=\la \Ev_g\otimes\eta\otimes\Ev_{g^{-1}},
(m\times\mathrm{id}_G)^*m^*(f)\ra,
\end{aligned}\]
for each $g\in G$ and $\eta\in\g$,
and 
\[\la\tau,f \ra=\lim_{t\rightarrow 0} \frac{f(\exp{(t\tau)})-f(e)}{t}=\lim_{t\rightarrow 0}\la \frac{\Ev_{\exp(t\tau)}-\Ev_e}{t},f\ra \quad\text{for each $\tau\in \Lie(\underline{G})$}.\]
Using these, it is easy to verify the two compatibility conditions
\eqref{eq:qLmodule1} and \eqref{eq:qLmodule2}. Then 
$\CO_M(M)$ is a $(\g,\underline{G})$-module, as desired.
\qed
\vspace{3mm}

We now prove the assertion (c) in Proposition \ref{prop:modofgG}.
 Recall from \cite[Definition 3.2]{CSW1} that for a commutative unital $\BC$-algebra $A$, an element $D\in \Hom_{\BC} (A,A)$ is a derivation if
\[D(b_1b_2)=D(b_1)b_2+b_1D(b_2)\quad \text{for all $b_1,b_2\in A$.}\]  By using \eqref{eq:Prim}, the following lemma is straightforward.
\begin{lemd}\label{lem:etader}
	  For each $\eta\in \g$, the map
	  \[D_{\eta}:\ \CO_M(M)\rightarrow \CO_M(M), \quad  f\mapsto\eta.f\]
	 is a derivation. 
\end{lemd} 


\begin{lemd}\label{lem:GOMsm}
	    The subspace $\RD^{\infty}_c(M;\CO_M)$ of $\RD^{-\infty}_c(M;\CO_M)$ is $\g$-stable. Moreover, the Lie algebra action 
\be\label{eq:etaonDc} \g\times \RD^{\infty}_c(M;\CO_M) \rightarrow \RD^{\infty}_c(M;\CO_M), \quad (\eta,\tau)\mapsto \eta.\tau\ee (as in Proposition \ref{prop:modofgG} (c)) is continuous.
\end{lemd}
\begin{proof}
Let $\eta\in \g$. 
For all $D\in \mathrm{Diff}_{K,r} (\CO_M,\underline{\mathcal{D}_M})$ with $(K,r)\in \mathcal{C}(M)$, it follows from  Lemma \ref{lem:etader} and \cite[Proposition 3.8]{CSW1} that $D\circ D_{-\eta}\in \mathrm{Diff}_{K,r+1}(\CO_M,\underline{\mathcal{D}_{M}})$. 
Then one has that \[
\begin{aligned}
\langle \eta.\rho_M(D),f \rangle
&= \langle \rho_M(D), -\eta. f \rangle
 = \langle \rho_M(D), D_{-\eta}(f) \rangle \\
&= \int_M (D\circ D_{-\eta})(f)
 = \langle \rho_M(D\circ D_{-\eta}), f \rangle 
\end{aligned}
\] for every $f\in \CO_M(M)$. 
This implies that  \be\label{eq:etaonD} \eta.\rho_M(D)=\rho_M(D\circ D_{-\eta})\in \RD^{\infty}_c(M;\CO_M),\ee and then the first assertion follows.

We now prove the second assertion. For all $D\in \mathrm{Diff}_{K,r} (\CO_M,\underline{\mathcal{D}_M})$ with $(K,r)\in \mathcal{C}(M)$, it is clear that 
\[
\abs{D\circ D_{-\eta}}_{X,f}=\sup_{a\in M}|\big((X\circ D)(-\eta.f)\big)(a)|=\abs{D}_{X,-\eta.f}
\]
for each $X\in \mathrm{Diff}_c(\underline{\mathcal{D}_M},\underline{\CO_M})$ and each $f\in \CO_M(M)$.  
This implies that the linear map
\[
\mathrm{Diff}_{K,r}(\CO_M,\underline{\mathcal{D}_M})
\rightarrow
\mathrm{Diff}_{K,r+1}(\CO_M,\underline{\mathcal{D}_M}),
\quad
D\mapsto D\circ D_{-\eta},
\]
is continuous. Hence, by \eqref{eq:etaonD},
\[
\RD^{\infty}_{K,r}(M;\CO_M)
\rightarrow
\RD^{\infty}_{K,r+1}(M;\CO_M),
\quad
\tau\mapsto\eta.\tau,
\]
is continuous. It then follows from \eqref{eq:D_ctop} that
\[
\RD^{\infty}_c(M;\CO_M)
\rightarrow
\RD^{\infty}_c(M;\CO_M),
\quad
\tau\mapsto\eta.\tau,
\]
is continuous. Together with \eqref{eq:algconti}, this proves the second assertion.
\end{proof}

Recall from \cite[Section 4.1]{CSW2} that the reduction 
\[\underline{\psi}:\ \underline{G}\times\underline{M}\rightarrow \underline{M}\] is a smooth action of the Lie group $\underline{G}$ on the smooth manifold $\underline{M}$. Then it is well-known that there is a smooth representation
\be \label{eq:Gondensity}\underline{G}\times \RD^{\infty}_c(M;\underline{\CO_M})\rightarrow \RD^{\infty}_c(M;\underline{\CO_M}), \quad (g,\tau)\mapsto g.\tau\ee of $\underline G$. Here, for each $g\in G$ and $\tau\in \RD^{\infty}_c(M;\underline{\CO_M})$, $g.\tau$ is the compactly supported smooth density such that 
\[\la g.\tau,f_0\ra=\la\tau,g^{-1}.f_0\ra \quad \text{(see \eqref{eq:GonOM})}\] for every $f_0\in \underline{\CO_M}(M)$.
In view of this, 
for each $g\in G$ and each $D 
\in\mathrm{Diff}_c(\CO_M,\underline{\mathcal{D}_M})$, write $g\circ D\circ g^{-1}$ for the linear map  \[ \CO_M(M)\rightarrow \underline{\mathcal{D}_M}(M), \quad f\mapsto  g.(D(g^{-1}.f)).\]

\begin{lemd}\label{lem:GOMsm1}
	    The subspace $\RD^{\infty}_c(M;\CO_M)$ of $\RD^{-\infty}_c(M;\CO_M)$ is $\underline{ G}$-stable. 
\end{lemd}
\begin{proof}
Let $(K,r)\in \mathcal{C}(M)$. 
For each $g\in G$ and each $D
\in\mathrm{Diff}_{K,r}(\CO_M,\underline{\mathcal{D}_M})$,  we claim that 
\be\label{eq:gDg-in} g\circ D\circ g^{-1}\in \mathrm{Diff}_{g.K,r}(\CO_M,\underline{\mathcal{D}_M}).\ee As usual, $g.K$ denotes the image of $\{g\}\times K$ under the continuous map $\overline{\psi}: G\times M\rightarrow M$. 
Then, for every $f\in \CO_M(M)$, 
\[
\begin{aligned}
&\la g.\rho_M(D),f\ra
=\la\rho_M(D),g^{-1}.f \ra
 = \int_M D(g^{-1}.f)  \\
=&\int_M (g^{-1}.1)\cdot D(g^{-1}.f)=\int_M 1 \cdot g.(D(g^{-1}.f))
 =\int_M (g\circ D\circ g^{-1})(f),
\end{aligned}
\] where $1\in \underline{\CO_M}(M)$.
This implies that \be \label{eq:gonD}g.\rho_M(D)=\rho_M(g\circ D\circ g^{-1})\in \RD^{\infty}_c(M;\CO_M)\ee for each $g\in G$ and each $D
\in\mathrm{Diff}_{K,r}(\CO_M,\underline{\mathcal{D}_M})$. 
The lemma then follows. 

Now we turn to prove \eqref{eq:gDg-in}.
Since $D\in \mathrm{Diff}_r(\CO_M,\underline{\mathcal{D}_M})$, by definition,
\[ [\cdots[[D,f_0],f_1],\cdots, f_{r}](f)=0, \quad\text{for all $f_0,f_1,\dots,f_r,f\in \CO_M(M)$}\]
(see \cite[Proposition 3.8 and Definition 3.1]{CSW1}).
This implies that 
\begin{eqnarray*}
   && [\cdots[[g\circ D\circ g^{-1},f_0],f_1],\cdots, f_{r}](f) \\ &=& g.([\cdots[[D,g^{-1}.f_0],g^{-1}.f_1],\cdots, g^{-1}.f_{r}](g^{-1}.f))=0
\end{eqnarray*}
for all $f_0,f_1,\dots, f_r,f\in \CO_M(M)$. Then 
$g\circ D\circ g^{-1}\in \mathrm{Diff}_r(\CO_M,\underline{\mathcal{D}_M})$. 

On the other hand, it is well-known that for each compactly supported smooth density $\tau$ whose support is contained in the compact subset $K$ of $\underline M$, the support of $g.\tau$ is contained in $g.K$. Then 
it is clear that \be\label{eq:supp}g\circ D\circ g^{-1}\in \mathrm{Diff}_{g.K}(\CO_M,\underline{\mathcal{D}_M})\ee and then \eqref{eq:gDg-in} follows, as desired.
\end{proof}

\begin{lemd}\label{lem:Df_i}
    Let $\{f_\alpha\}_{\alpha\in A}$ be a net in $\CO_M(M)$ such that $\lim_{\alpha\in A} f_{\alpha}=0$. Then for each $D\in \mathrm{Diff}_c(\CO_M,\underline{\mathcal{D}_M})$,
    \[\lim_{\alpha\in A} D(f_{\alpha})=0 \quad\text{in the LCS $\underline{\mathcal{D}_M}(M)$.}\]
\end{lemd}
\begin{proof}
    Note that for each $X\in \mathrm{Diff}_c(\underline{\mathcal{D}_M},\underline{\CO_M})$, $X\circ D\in\mathrm{Diff}_c(\CO_M,\underline {\CO_M}) $. Then the lemma follows from the definition of the smooth topologies on $\CO_M(M)$ and $\underline{\mathcal{D}_M}(M)$.
\end{proof}
\begin{lemd}\label{lem:gonDc}
    The group action  \be \label{eq:gonDc}G\times \RD^{\infty}_c(M;\CO_M)\rightarrow \RD^{\infty}_c(M;\CO_M),\quad (g,\tau)\mapsto g.\tau\ee is  continuous. 
\end{lemd}
\begin{proof}
 We claim that the map
\be \label{eq:GtoDc}G\rightarrow \RD^{\infty}_c(M;\CO_M),\quad g\mapsto g.\rho_M(D)=\rho_M(g\circ D\circ g^{-1})\ee is continuous for each $D\in \mathrm{Diff}_c(\CO_M,\underline{\mathcal{D}_M})$. Then
the map \eqref{eq:gonDc} is continuous by using the fact that for any group action \[\phi: L\times E \to E,\quad (g,u)\mapsto g.u,\] of a Lie group $L$ on a quasi-complete barreled LCS $E$,  
\be\label{eq:barreled}
\begin{aligned}
&\phi \text{ is continuous if and only if the map $(L\rightarrow E, g\mapsto g.u)$}  \\ & \text{is continuous for all $u\in E$ (\cf\cite[Page 219]{Ga}).}
\end{aligned}
\ee
It suffices to prove the claim.

Let $(K,r)\in \mathcal{C}(M)$ and let $D\in \mathrm{Diff}_{K,r}(\CO_M,\underline{\mathcal{D}_M})$. Let $\{g_\alpha\}_{\alpha\in A}$ be a net such that $\lim_{\alpha\in A} g_\alpha=e$. Then by Proposition \ref{prop:modofgG} (a),  for each $f\in \CO_M(M)$, \[\lim_{\alpha\in A}g^{-1}_\alpha.f=e.f=f\quad \text{in the LCS $\CO_M(M)$.}\]
This, together with  Lemma \ref{lem:Df_i}, implies that  
\[\lim_{\alpha\in A}D(g_{\alpha}^{-1}.f)=D(f)\quad \text{in the LCS $\underline{\mathcal{D}_M}(M)$}.\] It follows from \eqref{eq:Gondensity} and \eqref{eq:convergence} that 
\[\lim_{\alpha\in A} g_{\alpha}\circ D\circ g^{-1}_{\alpha}=D\quad\text{in the LCS $\mathrm{Diff}_{\mathrm{fin}}(\CO_M,\underline{\mathcal{D}_M})$}.\]
 By using \eqref{eq:gDg-in}, without loss of generality, we assume that 
\[\lim_{\alpha\in A} g_{\alpha}\circ D\circ g^{-1}_{\alpha}=D\quad\text{in $\mathrm{Diff}_{K_0,r}(\CO_M,\underline{\mathcal{D}_M})$}\]
for some $(K_0,r)\in \mathcal{C}(M)$. 
Then it is clear that 
\[\lim_{\alpha\in A} g_\alpha.\rho_M(D)=\rho_M(D) \quad \text{in  $\RD^{\infty}_{K_0,r}(M;\CO_M)$}.\]
This implies that the map \eqref{eq:GtoDc} is continuous, as desired. 
\end{proof}

\noindent\textbf{Proof of Proposition \ref{prop:modofgG} (c):} 
The first assertion follows from Lemmas \ref{lem:GOMsm} and \ref{lem:GOMsm1}. 
By Lemmas \ref{lem:GOMsm} and \ref{lem:gonDc}, the restricted actions of $\g$ and 
$\underline G$ on 
$\RD^\infty_c(M;\CO_M)$ are continuous. It suffices to verify the two compatibility conditions \eqref{eq:qLmodule1} and \eqref{eq:qLmodule2}.

By using  
Proposition \ref{prop:modofgG} (b), it is clear that 
for all $g\in G, \eta \in \g$ and $\tau\in \RD^{\infty}_c(M;\CO_M)$, \[ g.(\eta.(g^{-1}.\tau))=(\Ad_g \eta).\tau.\]
Let $(K,r)\in \mathcal{C}(M)$. 
We claim that for each $D\in \mathrm{Diff}_{K,r}(\CO_M,\underline{\mathcal{D}_M})$ and each $\eta\in \Lie(\underline{ G})$, 
\be\label{eq:etaD}\text{$\big(\eta\circ D:\ \CO_M(M)\rightarrow \underline{\mathcal{D}_M}(M),\quad f\mapsto \eta.(D(f))\big)\in \mathrm{Diff}_{K,r+1}(\CO_{M},\underline{\mathcal{D}_M})$}\ee
and 
\be\label{eq:diffgDg-1} \lim_{t\rightarrow 0} \frac{\exp(t\eta)\circ D\circ \exp(-t\eta)-D}{t}=\eta\circ D+D\circ D_{-\eta}\ee
in $\mathrm{Diff}_{K_0,r+1}(\CO_M,\underline{\mathcal{D}_M})$ for some $(K_0,r+1)\in \mathcal{C}(M)$. 
Note that for all $f\in \CO_M(M)$, 
\[ \la \rho_M(\eta\circ D),f\ra =\int_M \eta.(D(f))=\int_M (-\eta.1)\cdot D(f)=0,\]where $1\in \underline{\CO_M}(M)$.
This, together with  \eqref{eq:etaonD}, implies that 
 \[\lim_{t\rightarrow 0}\frac{\exp(t\eta).\rho_M(D)-\rho_M(D)}{t} =\eta.\rho_M(D)\quad \text{for each $\eta\in \Lie(\underline{ G})$}.\] Then the second assertion follows. In what follows, we prove \eqref{eq:etaD} and \eqref{eq:diffgDg-1}.

For \eqref{eq:etaD}, it is straightforward  that for each smooth density $\tau\in\underline{\mathcal{D}_M}(M)$ and each $h\in \CO_M(M)$,

\[\eta.(h\tau)=(\eta.h)\tau+h(\eta.\tau).\]
Here, for each $\CO_M$-module $\CF$, \[\CO_{M}(M)\times \CF(M)\rightarrow \CF(M), \quad (h,u)\mapsto hu\]
is the $\CO_M(M)$-module structure on $\CF(M)$. 
This, together with \cite[Proposition 3.8]{CSW1}, implies 
\[(\tau\mapsto \eta.\tau)\in \mathrm{Diff}_1(\underline{\mathcal{D}_M},\underline{\mathcal{D}_M})\] and then  \eqref{eq:etaD} is clear.

For \eqref{eq:diffgDg-1}, note that for each $f\in \CO_M(M)$ and each $t\in\BR\setminus\{0\}$, 
\be\label{eq:gD0} \begin{aligned}
    &\frac{(\exp(t\eta)\circ D\circ\exp(-t\eta) )(f)-D(f)}{t}\\=&\frac{\exp(t\eta).\big(D(\exp(-t\eta).f-f)\big)}{t}+\frac{\exp(t\eta).(D(f))-D(f)}{t}.
\end{aligned}\ee 
Using \eqref{eq:Gondensity}, it is clear that \[\lim_{t\rightarrow 0} \frac{\exp(t\eta).(D(f))-D(f)}{t}=\eta.(D(f)).\]
Then it follows from  \eqref{eq:convergence} and \eqref{eq:Gondensity} that 
\be\label{eq:gD1} \lim_{t\rightarrow 0 }\frac{\exp(t\eta)\circ D-D}{t}=\eta\circ D\ee
in $\mathrm{Diff}_{K_1,r+1}(\CO_M,\underline{\mathcal{D}_M})$ for some $(K_1,r+1)\in \mathcal{C}(M)$. 

On the other hand, by Proposition \ref{prop:modofgG} (a), we have that  
\[\lim_{t\rightarrow 0} \frac{\exp(-t\eta).f-f}{t}=-\eta.f \quad \text{in the LCS $\CO_M(M)$}.\] Then it follows from Lemma \ref{lem:Df_i} and \eqref{eq:Gondensity} that 
\[\lim_{t\rightarrow 0}\exp(t\eta).\big(\frac{D(\exp(-t\eta).f-f)}{t}\big)=D(-\eta.f). \] Hence, by \eqref{eq:convergence}, \eqref{eq:Gondensity} and \eqref{eq:gDg-in}, it is clear that 
\be\label{eq:gD2} \lim_{t\rightarrow 0}\frac{\exp(t\eta)\circ D\circ \exp(-t\eta)-\exp(t\eta)\circ D}{t}=D\circ D_{-\eta}\ee in $\mathrm{Diff}_{K_2,r+1}(\CO_M,\underline{\mathcal{D}_M})$ for some $(K_2,r+1)\in \mathcal{C}(M)$. 
Then \eqref{eq:diffgDg-1} follows from \eqref{eq:gD0},  \eqref{eq:gD1} and \eqref{eq:gD2} immediately. This finishes the proof. \qed

\section{Proof of Theorem  \ref{thm:eqMR}} 
In this section, we prove Theorem \ref{thm:eqMR} by Proposition \ref{prop:eqMR}.
\subsection{Functors between categories} 
Let $G$ be a formal Lie group.
Denote by 
\[\mathcal{C}\mathrm{om}_{\CO_G(G)},\quad{\mathcal{R}}\mathrm{ep}_{G},\qaq \mathcal{M}\mathrm{od}_{(\g,\underline{G})}\]
the categories of  $\CO_G(G)$-comodules, $G$-representations  and $(\g,\underline G)$-modules, respectively. 

Let $E$ be a quasi-complete LCS.
As usual, let $\RC^{\infty}(\underline{G};E)$  denote the space of $E$-valued smooth functions on $\underline G$. Then
it follows from \[E\wt\otimes_{\pi}\RC^{\infty}(\underline G)=\RC^{\infty}(\underline G;E)\quad \text{(see  \cite[Th\'{e}or\`eme 1]{Sc})}\] and Proposition \ref{prop:OofLLqLl} that 
there is an identification \be \label{eq:HOME}E\wt\otimes_{\pi}\CO_G(G)=\Hom_{\RU(\underline{\g})}(\RU(\g),\RC^{\infty}(\underline{G};E)) \ee as LCS.
Here $\RU(\underline{\g})$ acts on $\RU(\g)$ by left multiplication, and on $\RC^{\infty}(\underline G;E)$ via
\[\RU(\underline \g)\times E\wt\otimes_{\pi}\RC^{\infty}(\underline{G})\rightarrow E\wt\otimes_{\pi}\RC^{\infty}(\underline{G}),\quad
(\tau,u\otimes f)\mapsto u\otimes (g\mapsto \la \Ev_g\tau,f\ra
),
\]
and $\Hom_{\RU(\underline{\g})}(\RU(\g),\RC^{\infty}(\underline{G};E))$ is equipped with the term-wise convergence topology. 
As in \eqref{eq:Hom()}, we will often identify an element $f\in \Hom_{\RU(\underline{\g})}
(\RU(\g),\RC^\infty(\underline G;E))$ with a map
\[
f:\ \underline G\times\RU(\g)\longrightarrow E
\]
such that $f$ is smooth in the first variable, linear in the second variable, and
\[
f(g,\tau\eta)
=\la\Ev_g\tau,f(\,\cdot \,,\eta)\ra
\] for all $\tau\in\RU(\underline{\g})$, $g\in G$, and $\eta\in\RU(\g)$. Here, $f(\, \cdot\,,\eta)$ is the $E$-valued smooth function 
$x\mapsto f(x,\eta)$ on $\underline{G}$.

\begin{prpd}\label{prop:eqMR}
\noindent (a) Let $(\rho,E)$ be a comodule of $\CO_G(G)$. Then  $E$ together with the linear map 
    \be\label{eq:mu*} \rho^{\circ}: \ \RD^{-\infty}_c(G;\CO_G)\wt\otimes_{\mathrm{i}} E\rightarrow E,\quad  \eta\otimes u\mapsto \la\eta,\rho(u)\ra_2\quad\text{(see \eqref{eq:ra12})}\ee is a $G$-representation. Furthermore, the assignment 
    \[(\rho,E)\mapsto(\rho^{\circ},E)\] forms a functor 
    \[F_1:\ \mathcal{C}\mathrm{om}_{\CO_G(G)} \rightarrow\mathcal{R}\mathrm{ep}_G.\]
    
\noindent (b)   Let $(\pi,E)$ be a $G$-representation. Then $E$ together with the restriction maps \be \label{eq:dvarrho}\pi|_{\g}:\ \g\times E\rightarrow E, \quad (\eta,u)\mapsto \eta.u\ee and 
\be \label{eq:linevarrho}\pi|_{G}:\ \underline G\times E\rightarrow E, \quad (g,u)\mapsto g.u:=\Ev_g.u\quad\text{(see Corollary \ref{cor:Gembed})}\ee  is a $(\g,\underline G)$-module. 
 Furthermore, the assignment
    \[(\pi,E)\mapsto((\pi|_{\g},\pi|_{G}),E) 
    \] forms a functor 
    \[F_2:\ \mathcal{R}\mathrm{ep}_G \rightarrow \mathcal{M}\mathrm{od}_{(\g,\underline{G})}.\]

 \noindent  (c)   Let $(\mu = (\mu_{\g}, \mu_{\underline{G}}), E)$ be a $(\mathfrak{g}, \underline{G})$-module. 
    Then $E$ together with the linear map  
     \[
     \begin{array}{rcl}
         \mu^*:\ E &\rightarrow &E \wt{\otimes}_{\pi} \mathcal{O}_G(G)=\Hom_{\RU(\underline{\g})}(\RU(\g),\RC^{\infty}(\underline G;E)),\\  u &\mapsto &\big((g,\eta)\mapsto g.(\eta.u)\big)
     \end{array}
    \] is a comodule of $\mathcal{O}_G(G)$. Furthermore, the assignment
    \begin{equation}\label{eq:assEtogE2}
        (\mu, E) \mapsto (\mu^*, E)
    \end{equation}
    is a functor 
    \[F_3:\ \mathcal{M}\mathrm{od}_{(\g,\underline{G})}\rightarrow \mathcal{C}\mathrm{om}_{\CO_G(G)}.\]

  \noindent(d)
   We have that 
    \[
       F_3 \circ F_2\circ F_1=\mathrm{Id}_{\mathcal{C}\mathrm{om}_{\CO_G(G)}},\ F_1 \circ F_3\circ F_2= \mathrm{Id}_{{\mathcal{R}}\mathrm{ep}_{G}},\ F_2 \circ F_1\circ F_3=\mathrm{Id}_{\mathcal{M}\mathrm{od}_{(\g,\underline{G})}},
    \]  where for each category $\mathcal{C}$, $\mathrm{Id}_{\mathcal{C}}$ is the identity functor. 
\end{prpd}

\subsection{Proof of Proposition \ref{prop:eqMR}} In this subsection, we prove the assertions in Proposition \ref{prop:eqMR} successively.

\noindent\textbf{Proof of Proposition \ref{prop:eqMR} (a):}
 The second assertion follows from the first assertion. Now we prove the first assertion. 
 
 By using \eqref{eq:RepchianruleofO}, we have that $\rho^{\circ}$ satisfies  \eqref{eq:RepchianruleofD} (replace $\pi$ by $\rho^{\circ}$). 
Using the same argument as in the proof of the continuity of the map \eqref{eq:wttimes2} in Lemma \ref{lem:GcontinuonOM}, it is clear that the induced map  \be\label{eq:GonEfunctor2} G\times E\rightarrow E, \quad (g,u)\mapsto \la \Ev_g,\rho(u)\ra_2\ee is continuous.
It suffices to prove that the map $\rho^{\circ}$ is continuous, 
namely, the bilinear map 
        \be \label{eq:mu*2} \RD^{-\infty}_c(G;\CO_G)\times E\rightarrow E,\quad  (\eta, u)\mapsto \la\eta,\rho(u)\ra_2\ee
 is separately continuous. 
        
        Let $\eta_0\in \RD^{-\infty}_c(G;\CO_G)$.
        Note that the 
        linear map \be\label{eq:mu*eta0} E\rightarrow E,\quad u\mapsto \la\eta_0,\rho(u) \ra_2\ee is the   composition 
        \[E\xrightarrow{\rho}E\wt\otimes_{\pi}\CO_G(G)\xrightarrow{\mathrm{id}_{E}\otimes\eta_0}E\wt\otimes_{\pi}\BC=E\] of continuous maps.  Therefore the map \eqref{eq:mu*eta0} is continuous.

On the other hand, for every $u\in E$, 
write 
\be\label{eq:fu}
f_{u}\in
\Hom_{\RU(\underline{\g})}
\bigl(\RU(\g),\RC^\infty(\underline G;E)\bigr)
\ee for the element 
corresponding to $\rho(u)\in E\wt\otimes_{\pi}\CO_G(G)$ (see \eqref{eq:HOME}). Let $u_0\in E$. By \eqref{eq:D=DU} and \eqref{eq:pairing2}, the continuity of the linear map 
\[
\RD^{-\infty}_c(G;\CO_G)\rightarrow E,
\qquad
\eta\mapsto\la\eta,\rho(u_0)\ra_2
\]
follows from the separate continuity of the bilinear map 
\[
\RD^{-\infty}_c(\underline G)\times\RU(\g)\longrightarrow E,
\qquad
(\tau,\eta)\longmapsto
\la\tau,f_{u_0}(\,\cdot\,,\eta)\ra,
\]
which is immediate from \eqref{eq:fu}. This finishes the proof.
\qed

\vspace{3mm}

\noindent\textbf{Proof of Proposition \ref{prop:eqMR} (b):}
The second assertion follows from the first assertion. Now we prove the first assertion. 

By using the continuity of \(\pi\), Corollary \ref{cor:Gembed} and \eqref{eq:RepchianruleofD}, it is clear that \eqref{eq:dvarrho} is a continuous Lie algebra action of $\g$ and \eqref{eq:linevarrho} is a continuous group action of $\underline G$.

It remains to verify the two compatibility conditions in Definition \ref{def:ofqLmod}.
By \eqref{eq:RepchianruleofD} and Corollary \ref{cor:realAd}, we have that 
\[g.(\eta.(g^{-1}.u))=(\Ad_g\eta).u
\]
for all \(g\in G\), \(\eta\in\g\), and \(u\in E\). On the other hand,  we claim that \be\label{eq:limtau} \lim_{t\rightarrow 0} \frac{\Ev_{\exp(t\tau)}-\Ev_{e}}{t}=\tau\quad\text{(in $\RD^{-\infty}_c(G;\CO_G)$) }\ee for each $\tau\in\Lie(\underline G)$. By using this claim, it is clear that for each $u\in E$, 
 \[\lim_{t\rightarrow 0}\frac{\exp(t\tau).u-u}{t}=(\lim_{t\rightarrow 0} \frac{\Ev_{\exp(t\tau)}-\Ev_{e}}{t}).u=\tau.u.\] Then the first assertion follows.

In what follows,  we will prove \eqref{eq:limtau}. Note that there is a one-parameter subgroup
\[ \exp_{\tau}=(\overline{\exp_{\tau}},\exp_{\tau}^*):\ \R^{(0)}\rightarrow  G\]  given by the continuous homomorphism \[\exp_{\tau}^*:\ \CO_G(G)\rightarrow \CO_{\R}^{(0)}(\R),\quad f\mapsto (t\mapsto f(\exp(t\tau))).\]
 Then for each bounded subset $B$ of $\CO_G(G)$ and $t\neq 0$, we have that  
\begin{eqnarray*}
    \sup_{f\in B} |\la \frac{\Ev_{\exp(t\tau)}-\Ev_{e}}{t}-\tau, f\ra| &=&\sup_{f\in B} |\frac{f(\exp(t\tau))-f(e)}{t}-\la \tau, f\ra|\\
    &=& \sup_{h\in B_0}  |\frac{h(t)-h(0)}{t}-h'(0)|.
\end{eqnarray*}
Here, $B_0:=\exp_{\tau}^*(B)$ is a bounded subset of  $\CO_{\R}^{(0)}(\R)$.
It follows from the mean value theorem in several variables (see \cite[Theorem 12.9]{Ap})  that 
\[\lim_{t\rightarrow 0}\sup_{h\in B_0}  |\frac{h(t)-h(0)}{t}-h'(0)|=0.\]
Hence, for each bounded subset $B$ of $\CO_G(G)$,
\[  \lim_{t\rightarrow 0}\sup_{f\in B} |\la \frac{\Ev_{\exp(t\tau)}-\Ev_{e}}{t}-\tau, f\ra|=\lim_{t\rightarrow 0}  \sup_{f\in B_0}  |\frac{f(t)-f(0)}{t}-f'(0)|=0,\] and
then \eqref{eq:limtau} follows. This finishes the proof.
\qed

\vspace{3mm}

\noindent\textbf{Proof of Proposition \ref{prop:eqMR} (c):} 
 The second assertion follows from the first assertion. Now we prove the first assertion. 

For each $u\in E$ and $\eta \in \g$, it is clear that 
\[\big(g\mapsto g.(\eta.u)\big)\in \RC^{\infty}(\underline G;E).\] 
This implies that 
\[\big((g,\eta)\mapsto g.(\eta.u)\big)\in \Hom_{\BC}(\RU(\g),\RC^{\infty}(\underline G;E)). \]
On the other hand, note that for $g\in \underline{G}$,  $\tau\in \Lie(\underline{G})$, $\eta\in \g$, and $u\in E$,
\[\lim_{t\rightarrow 0}\frac{g.(\exp(t\tau).(\eta.u))-g.(\eta.u)}{t}=g.( \lim_{t\rightarrow 0}\frac{\exp(t\tau).(\eta.u)-\eta.u}{t})=g.(\tau.(\eta.u))= g.((\tau\eta).u).\] Therefore, the map 
     \[\mu^*:\ E \longrightarrow E \wt{\otimes}_{\pi}\mathcal{O}_G(G)=\Hom_{\RU(\underline{\g})}(\RU(\g),\RC^{\infty}(\underline G;E)),\quad u\mapsto ((g,\eta)\mapsto g.(\eta.u))
    \] is well-defined. Note that for each $\eta\in \g$,  the orbit map \[E\rightarrow \RC^{\infty}(\underline G;E),\quad u\mapsto \big(g\mapsto g.(\eta.u)\big)\] is continuous.
     Then $\mu^*$ is continuous. 
     It suffices to prove that the diagram 
\be\label{eq:diafortoEO}\begin{CD}
E @>  \mu^*>>  E\wt\otimes_{\pi} \CO_G(G)\\
	@V\mu^* VV           @V V \mu^* \otimes \mathrm{id}_{\CO_G(G)} V\\
		E\wt\otimes_{\pi}\CO_G(G) @>\mathrm{id}_{E}\otimes m^*>> E\wt\otimes_{\pi} \CO_G(G)\wt\otimes_{\pi}\CO_G(G),
\end{CD} \ee commutes.

By Proposition \ref{prop:OofLLqLl} and \eqref{eq:HOME}, we have that \begin{eqnarray*}E\wt\otimes_{\pi} \CO_G(G)\wt\otimes_{\pi} \CO_G(G) &=&E\wo_{\pi}\Hom_{\RU(\underline{\g})}(\RU(\g),\RC^{\infty}(\underline{G}))\wo_{\pi}\Hom_{\RU(\underline{\g})}(\RU(\g),\RC^{\infty}(\underline{G}))\\&\subset& E\wo_{\pi}\Hom_{\BC}(\RU(\g),\RC^{\infty}(\underline{G}))\wo_{\pi}\Hom_{\BC}(\RU(\g),\RC^{\infty}(\underline{G}))\\ &=&\Hom_{\BC}(\RU(\g)\otimes \RU(\g),\RC^{\infty}(\underline{G}\times\underline{G};E)).\end{eqnarray*}
In view of this, we will also view an element in 
$E\wt\otimes_{\pi} \CO_G(G)\wt\otimes_{\pi} \CO_G(G) $ as  a map \[(\underline{G}\times\underline{G})\times(\RU(\g)\otimes \RU(\g))\rightarrow E.\]  
By Proposition \ref{prop:OofLLqLl}, it is straightforward that  for each $u\in E$, 
\begin{eqnarray*}
    ((\mathrm{id}_E\otimes m^*)\circ\mu^*)(u)&=&\big((g_1,g_2,\eta_1\otimes\eta_2)\mapsto g_1.(g_2.((\Ad_{g_2^{-1}}(\eta_1)).(\eta_2.u)))\big)\\
    &=&\big((g_1,g_2,\eta_1\otimes\eta_2)\mapsto g_1.(\eta_1.(g_2.(\eta_2.u)))\big)\quad \text{(see Corollary \ref{cor:realAd})}.
\end{eqnarray*}
We also have that  for each $u\in E$, 
\[((\mu^* \otimes \mathrm{id}_{\CO_G(G)} )\circ \mu^*)(u)=\big((g_1,g_2,\eta_1\otimes\eta_2)\mapsto g_1.(\eta_1.(g_2.(\eta_2.u)))\big).\] Then the diagram \eqref{eq:diafortoEO} commutes, as desired.\qed

\vspace{3mm}
\noindent\textbf{Proof of Proposition \ref{prop:eqMR} (d):} The assertion (d) is straightforward by using Proposition \ref{prop:OofLLqLl}.
\qed

\section*{Acknowledgement}
Fulin Chen was supported by the Natural Science Foundation of Xiamen, China (Grant No. 3502Z202473005), the National Natural Science Foundation of China (Grant No. 12471029), and the Natural Science Foundation of Fujian Province, China (Grant No. 2026J001042). Binyong Sun was supported in part by National Key R \& D Program of China (Grant No. 2022YFA1005300) and New Cornerstone
Science Foundation. 


\begin{thebibliography}{99}



\bibitem[A]{Ap}
T.~M. Apostol, \textit{ Mathematical Analysis}, second edition, 
Addison-Wesley Publishing Co., Reading, Mass.-London-Don Mills, Ont., 1974



\bibitem[C]{Ca}
B. Casselman, \textit{Quasi-complete topological vector spaces}, https://personal.math.ubc.ca/~cass/research/pdf/QC.pdf

\bibitem[CSW1]{CSW1} F. Chen, B. Sun and C. Wang, \textit{Formal manifolds: foundations}, Sci. China Math. {\bf 69} (2026), no.~1, 183--216.
\bibitem[CSW2]{CSW2} F. Chen, B. Sun and C. Wang, \textit{Function spaces on formal manifolds}, arXiv:2407.09329, to appear in Chinese Ann. Math. Ser. B.


\bibitem[CSW3]{CSW3} F. Chen, B. Sun and C. Wang,
\textit{Formal manifolds: local structure of morphisms, and formal submanifolds},
Acta Math. Sin. (Engl. Ser.) {\bf 42} (2026), no.~3, 603--647.

\bibitem[CSW4]{CSW4} F. Chen, B. Sun and C. Wang,
 \textit{Lie pairs and formal Lie groups}, arXiv:2604.25616

 
\bibitem[G]{G}
H.~Grobner, \textit{Smooth-Automorphic Forms and Smooth-Automorphic Representations},
Series on Number Theory and Its Applications, 17, World Scientific, Singapore, 2023.









\bibitem[KV]{KV} A.~W. Knapp and D.~A. Vogan Jr, \textit{Cohomological Induction and Unitary Representations}, Princeton Mathematical Series, 45, Princeton Univ. Press, Princeton, NJ, 1995.

\bibitem[Me]{Me}
R.~Meyer,
\textit{Analytic cyclic cohomology},
Ph.D. thesis, Universit\"at M\"unster, 1999,
arXiv:math/9906205.



\bibitem[Mi]{Mi}
P.~W. Michor, \textit{Radon transform and curvature}, in 75 years of Radon transform (Vienna, 1992), 249--251, Conf. Proc. Lecture Notes Math. Phys., IV, Int. Press, Cambridge, MA, 1994




%




\bibitem[S]{Sc} L. Schwartz, \textit{Espaces de fonctions diff\'erentiables \`a valeurs vectorielles}, Journal d'Analyse, 4, 1954,
88-148.







\bibitem[T]{Tr}
F. Tr\`eves, \textit{Topological Vector Spaces, Distributions and
Kernels}, Academic Press, New York, 1970.


%
%
\bibitem[W]{Ga} 
G. Warner, \textit{ Harmonic analysis on semi-simple Lie groups. I}, Die Grundlehren der mathematischen Wissenschaften, Band 188, Springer, New York-Heidelberg, 1972.
 
\bibitem[Wa]{Wa}
N.~R. Wallach, \textit {Real reductive groups. I}, Pure and Applied Mathematics, 132, Academic Press, Boston, MA, 1988.






\end{thebibliography}
\end{document}